\documentclass{amsart}
\usepackage{caption}
\usepackage{color}

\newtheorem{thm}{Theorem}
\newtheorem{lem}[thm]{Lemma}

\newtheorem{prop}[thm]{Proposition}
\newtheorem{set}[thm]{Setting}

\newcommand{\vip}{\vskip.2cm}
\newcommand{\e}{{\varepsilon}}
\newcommand{\rr}{{\mathbb{R}}}
\newcommand{\rd}{{\rr^d}}
\newcommand{\nn}{{\mathbb{N}}}
\newcommand{\E}{\mathbb{E}}
\newcommand{\Var}{\mathbb{V}{\rm ar}\,}
\newcommand{\cH}{{\mathcal H}}
\newcommand{\cN}{{\mathcal N}}
\newcommand{\cM}{{\mathcal M}}
\newcommand{\cW}{{\mathcal W}}
\newcommand{\cB}{{\mathcal B}}
\newcommand{\cT}{{\mathcal T}}
\newcommand{\cS}{{\mathcal S}}
\newcommand{\cP}{{\mathcal P}}

\newcommand{\cQ}{{\mathcal Q}}
\newcommand{\cR}{{\mathcal R}}

\newcommand{\intrd}{{\int_{\rd}}}
\newcommand{\dd}{{\rm d}}
\newcommand{\mm}{m}

\newcommand{\bla}{\color{black}}
\newcommand{\blu}{\color{black}}

\begin{document}

\title[Convergence of the empirical measure]
{Convergence of the empirical measure in expected Wasserstein distance: non asymptotic explicit bounds in $\rr^d$}

\author{Nicolas Fournier}

\address{N. Fournier, Sorbonne Universit\'e, CNRS, Laboratoire de Probabilit\'es, Statistiques 
et Mod\'elisation, F-75005 Paris, France. E-mail: {\tt nicolas.fournier@sorbonne-universite.fr}}

\begin{abstract}
We provide some
non asymptotic bounds, with explicit constants, 
that measure the rate of convergence, in expected 
Wasserstein distance, of the empirical measure associated to an i.i.d. $N$-sample of a given
probability distribution on $\mathbb{R}^d$.
\end{abstract}

\subjclass[2010]{60F25, 65C05}

\keywords{Empirical measure, Sequence of i.i.d. random variables, Optimal transportation}

\thanks{I thank J.L. Verger-Gaugry and T. Le Gouic for fruitful discussions and
the anonymous referee for his comments}
\maketitle

\section{Introduction}
Let $d\geq 1$. 
We consider $\mu \in \cP(\rd)$, the set of probability measures on $\rd$,
and an i.i.d. sequence $(X_k)_{k\geq 1}$
of $\mu$-distributed random variables. For $N\geq 1$, we introduce the empirical measure
\begin{equation}\label{muN}
\mu_N=\frac1N
\sum_{k=1}^N \delta_{X_k}.
\end{equation}

Estimating the rate of convergence of $\mu_N$ to $\mu$ is of course a fundamental problem,
and it seems that measuring this convergence in Wasserstein distance
is nowadays a widely adopted choice. 
Some seminal works on the subject are those by Dudley \cite{d},
Ajtai-Koml\'os-Tusn\'ady \cite{akt} and Dobri\'c-Yukich \cite{dy}. More recently, some results
have been established by Bolley-Guillin-Villani \cite{bgv}, Boissard-Le Gouic \cite{blg},
\blu Le Gouic \cite{lg}, \bla
Dereich-Scheutzow-Schottstedt \cite{dss} and Fournier-Guillin \cite{fg}.
In particular, we can find in \blu \cite{fg} \bla the following result.

\vip

Fix some norm $|\cdot|$ on $\rr^d$ and consider, for $p>0$ and $\mu,\nu \in \cP(\rd)$, the transport cost
$$
\cT_p(\mu,\nu) = \inf\Big\{\int_{\rd\times\rd} |x-y|^p \xi(\dd x,\dd y) : \xi\in\cH(\mu,\nu)\Big\},
$$ 
where
$\cH(\mu,\nu)$ stands for the set of probability measures on $\rd\times\rd$ with marginals
$\mu$ and $\nu$. It holds that $\cT_p=\cW_p^p$, with the usual notation, if $p\geq 1$.
For $q>0$, we also define
$$
\cM_q(\mu)= \intrd |x|^q \mu(\dd x).
$$
There exists a constant $C_{d,p,q}$ such that for all $\mu \in \cP(\rr^d)$, for all $N\geq 1$,
with $\mu_N$ defined in \eqref{muN},
$$
\E[\cT_p(\mu_N,\mu)] \leq C_{d,p,q} [\cM_q(\mu)]^{p/q}\times \left\{ \begin{array}{cl}
N^{-1/2} & \hbox{if $p>d/2$ and $q>2p$},\\[3pt]
N^{-1/2} \log (1+N) & \hbox{if $p=d/2$ and $q>2p$},\\[3pt]
N^{-p/d} & \hbox{if $p\in (0,d/2)$ and $q>dp/(d-p)$}.
\end{array}\right.
$$
\blu This bound is sharp, as well as the number of required moments, 
but the constant $C_{d,p,q}$ is not explicit.
Some explicit constants are provided in \cite{dss} when $p>d/2$, but they are quite large.
One can also get some explicit bounds using \cite{lg}, see a few lines below.
\bla

\vip

\blu It seems that applied scientists really need some explicit values for the constant
$C_{d,p,q}$. \bla If following the proofs in \cite{dss,fg}, one finds 
some \blu rather large \bla constants.
But revisiting these proofs and optimizing as often as possible the computations,
which is the purpose of the present paper,
we obtain some rather reasonable constants, when using the maximum norm $|\cdot|_\infty$ on $\rd$. 
In particular, they remain finite
as the dimension tends to infinity (but of course, the rate of convergence in $N^{-p/d}$ is worse and worse).
\vip
The reason why the maximum norm $|\cdot|_\infty$ is used in \cite{dss,fg} it that the proofs rely
on a partitioning of the unit ball, and that a cube is very easy to cut into 
smaller cubes. We of course deduce some bounds for the more natural Euclidean norm $|\cdot|_2$,
multiplying the constant by $d^{p/2}$. This leads to a constant (for the Euclidean norm) 
that explodes as $d\to \infty$.
\vip
Using similar arguments, together with some ideas found in \blu Boissard-Le Gouic \cite{blg} and \bla 
Weed-Bach \cite{wb}, Lei \cite[Theorem 3.1]{l}
proves that one can also find, in the case of the Euclidean norm, 
some constants (that he does not make explicit) that remain finite as $d\to\infty$.
We also produce, in the present paper, some explicit constants in this context.
\vip

\blu Finally, let us mention that Dudley \cite{d} and more recently 
Boissard-Le Gouic \cite{blg}, Le Gouic \cite{lg} and
Weed-Bach \cite{wb} study the very interesting problem
of obtaining some rates of convergence depending on the {\it true} dimension of the problem:
if e.g. $\mu$ is a measure on $\rd$ but is actually carried by a manifold of lower dimension, what about
the rate of convergence ?
They introduce some notion of dimension of the measure $\mu$ and get some bounds
of $\E[\cT_p(\mu_N,\mu)]$ in terms of this dimension.
Let us mention the following formula, that can be found in the work of Le Gouic \cite[Theorem 3.2]{lg}
(after correcting two small mistakes, with the agreement of the author):
for any metric space $(E,d)$ with finite diameter $D$, for $\mu$ a probability measure on $E$ and 
for $\mu_N$ the associated empirical measure, it holds that for all $N\geq 1$, all $p\geq 1$, all $k\geq 0$,
$$
\E[\cT_p(\mu_N,\mu)]\leq 2^{2p-1}\Big[D^p 2^{1-kp}+ \frac{2^p}{\sqrt N}\int_{D^p2^{-(k+1)p}}^{D^p2^{-p}}
\sqrt{\cN(E,u^{1/p})-1} \; \dd u\Big],
$$
where $\cN(E,\e)$ is the minimal number of balls of radius $\e$ required to cover $E$.
This is a very deep and elegant formula. Let us mention that when applied to the case where
$E=B(0,1)$ in $\rr^d$, after optimizing in $k$, 
this produces some good bounds, with some constants that are a little greater than what 
we will find below. Le Gouic \cite{lg} also studies the non compact case $E=\rr^d$, without really tracking
the constants.

\bla
 
\vip
We refer to the introductions \blu of \cite{blg, dss,fg,lg,wb} \bla for some much more detailed presentations
of the subject and its numerous applications. \blu Let us also mention the closely related topic of 
{\it optimal matching}, see Barthe-Bordenave \cite{bb} and the references therein. \bla

\vip

\blu
Let us emphasize that the present paper contains no new idea: all the deep arguments 
have been previously introduced in the above mentioned papers. 
We only try to handle some slightly more precise computations.

\bla

\section{Main results}
\subsection{Basic notation}
\blu
For $\mm\in [1,\infty)$ and $x=(x_1,\dots,x_d)\in\rr^d$, we set
$$
|x|_\mm=\Big(\sum_{i=1}^d |x_i|^\mm \Big)^{1/\mm} 
\qquad \hbox{and} \qquad |x|_\infty=\max \{|x_1|,\dots,|x_d|\}.
$$
For $p>0$, for $\mu,\nu$ in $\cP(\rd)$ and for $\mm \in [1,\infty)\cup\{\infty\}$, we set
$$
\cT_p^{(\mm)}(\mu,\nu) = \inf \Big\{\int_{\rd\times\rd} |x-y|_\mm^p \xi(\dd x,\dd y) : \xi \in \cH(\mu,\nu)\Big\},
$$
where $\cH(\mu,\nu)$ is the set of probability measures on $\rd\times\rd$ with marginals
$\mu$ and $\nu$.
For $q>0$, for $\mu\in\cP(\rd)$ and for $\mm \in [1,\infty)\cup\{\infty\}$, we define
$$
\cM_q^{(\mm)}(\mu)= \intrd |x|_\mm^q \mu(\dd x).
$$
We of course have, since $|\cdot|_\infty \leq |\cdot|_\mm \leq d^{1/\mm} |\cdot|_\infty$,
\begin{equation}\label{neq}
\cT_p^{(\infty)}(\mu,\nu)\leq \cT_p^{(\mm)}(\mu,\nu)\leq d^{p/\mm} \cT_p^{(\infty)}(\mu,\nu)
\quad \hbox{and}\quad \cM_q^{(\infty)}(\mu)\leq \cM_q^{(\mm)}(\mu)\leq d^{q/\mm}\cM_q^{(\infty)}(\mu).
\end{equation}
\bla

\subsection{Covering number}
Our proofs are based on a suitable partitioning of the unit ball. The case
of the maximum norm is not hard, because it is easy to cut a cube into smaller cubes.
\blu The other cases are more intricate. For $\e\in (0,1]$ and $\mm \in [1,\infty)$, we define
\begin{equation}\label{Nr}
N_\e^{(\mm)}=\min\Big\{k\in \nn : \exists\; x_1,\dots,x_k \in B_\mm(0,1) \hbox{ such that } B_\mm(0,1)
\subset \cup_{i=1}^k B_\mm(x_i,\e)\Big\},
\end{equation}
where $B_\mm(x,\e)=\{y\in \rd : |x-y|_\mm<\e\}$, as well as
\begin{equation}\label{Kd}
K_d^{(\mm)}=\sup_{\e\in (0,1]}\e^d N_\e^{(\mm)}, \quad \hbox{so that for all } \e\in (0,1], \quad N_\e^{(\mm)} 
\leq K_d^{(\mm)} \,\e^{-d}.
\end{equation}
See \eqref{Kdc} and \eqref{KdVG} below for some estimates of these covering numbers.
\bla

\subsection{Main result}

For $x>0$ and $q>s>0$, we set
\begin{equation}\label{H}
H(x,s,q)=\Big(x\frac{q-s}{s}+(1+x)\Big(\frac{q}s\Big)^{q/(q-s)}\Big)^{s/q}\frac{q}{q-s}.
\end{equation}
\blu Observe that for each $x>0$, each $s>0$, $\lim_{q\to \infty} H(x,s,q)=1$.
The following formulas, that constitute the main results of the paper, are a little complicated,
but rather easy to calculate explicitly with a computer. \bla

\begin{thm}\label{mr}
We fix $p>0$ and $\mu \in \cP(\rd)$, we set 
$\e_p=\max\{2^{-1},2^{-p}\}$ and we recall \eqref{muN}. 

\vip
\blu
(i) If $p>d/2$ and $q>2p$, then  for $\mm \in [1,\infty)\cup\{\infty\}$, for all $N \geq 1$,
$$
\E[\cT_p^{(\mm)}(\mu_N,\mu)]\leq 2^p \frac{\kappa_{d,p}^{(\mm)}}{\sqrt N}[\cM_q^{(\mm)}(\mu)]^{p/q}
\theta_{d,p,q}^{(\mm)},
\quad \hbox{where} \quad
\theta_{d,p,q}^{(\mm)}=H\Big(\frac{\e_p}{\kappa_{d,p}^{(\mm)}},2p,q\Big),
$$
with
$$
\kappa_{d,p}^{(\infty)}=\frac{2^{d/2-1}}{1-2^{d/2-p}},
$$
and with, if $\mm \in [1,\infty)$,
$$
\kappa_{d,p}^{(\mm)}=\min\{\kappa_{d,p,r}^{(\mm)} : r  \geq 2\},\qquad \hbox{where} \qquad
\kappa_{d,p,r}^{(\mm)}=\sqrt{K_d^{(\mm)}}\frac{2^{p-1-d/2}r^{p+d/2}}{(r-1)^p(1-r^{d/2-p})}.
$$

(ii) If $p=d/2$ and $q>2p$, then for $\mm \in [1,\infty)\cup\{\infty\}$, for all $N \geq 1$,
$$
\E[\cT_p^{(\mm)}(\mu_N,\mu)]\leq 2^p\frac{\kappa_{d,p,N}^{(\mm)}}{\sqrt N}[\cM_q^{(\mm)}(\mu)]^{p/q}
\theta_{d,p,q,N}^{(\mm)}, \quad \hbox{where}\quad
\theta_{d,p,q,N}^{(\mm)}=H\Big(\frac{\e_p}{\kappa_{d,p,N}^{(\mm)}},2p,q\Big),
$$
with (here $\log_+(x)=\max\{\log x,0\}$)
$$
\kappa_{d,p,N}^{(\infty)}
=\frac{2^{p-1}}{p\log 2} \log_+\Big((2^{1-p}-2^{1-2p})\sqrt N \Big)+ \frac{2^{p-1}}{1-2^{-p}},
$$
and with, if $\mm \in [1,\infty)$,
$$\kappa_{d,p,N}^{(\mm)}=\min\{\kappa_{d,p,N,r}^{(\mm)} : r\geq 2\},$$ 
where
$$ 
\kappa_{d,p,N,r}^{(\mm)}=\frac{\sqrt{K_d^{(\mm)}}}2 \frac{r^{2p}}{(r-1)^p p\log r} 
\log_+\Big(2^{p+1}(r^{-p}-r^{-2p})\sqrt{\frac{N}{K_d^{(\mm)}}} \Big)
+ \frac{\sqrt{K_d^{(\mm)}}}2 \frac{r^{3p}}{(r-1)^p (r^p-1)}.
$$

(iii) If $p\in (0,d/2)$ and $q>dp/(d-p)$, then for $\mm \in [1,\infty)\cup\{\infty\}$, for all $N \geq 1$,
$$
\E[\cT_p^{(\mm)}(\mu_N,\mu)]\leq 2^p \frac{\kappa_{d,p}^{(\mm)}}{N^{p/d}} [\cM_q^{(\mm)}(\mu)]^{p/q}\theta_{d,p,q}^{(\mm)},
\quad \hbox{where}\quad
\theta_{d,p,q}^{(\mm)}=H\Big(\frac{2^{1-2p/d}\e_p}{\kappa_{d,p}^{(\mm)}},\frac{dp}{d-p},q \Big),
$$ 
with
$$
\kappa_{d,p}^{(\infty)}=\frac{2^{p-2p/d}(1-2^{-d/2})^{1-2p/d}}{1-2^{p-d/2}},
$$
and with, if $\mm \in [1,\infty)$,
$$
\kappa_{d,p}^{(\mm)}=\min\{\kappa_{d,p,r}^{(\mm)} : r \geq 2\},\qquad\hbox{where}\qquad 
\kappa_{d,p,r}^{(\mm)}=\Big(\frac{K_d^{(\mm)}}{4}\Big)^{p/d} \frac{r^{2p} (1-r^{-d/2})^{1-2p/d}}{(r-1)^p(1-r^{p-d/2})}.
$$
\bla
\end{thm}

\subsection{Comments}
By invariance by translation, we can replace \blu $\cM_q^{(\mm)}(\mu)$, \bla
in all the formulas, by
$$
\inf \Big\{\intrd \blu |x-x_0|_\mm^q \bla \mu(\dd x) : x_0 \in \rd \Big\}.
$$

Let us next observe, and we will see that this is often advantageous, that
concerning the bound of $\E[\cT_p^{(\mm)}(\mu_N,\mu)]$ \blu when $\mm \in [1,\infty)$, \bla
we can replace, by \eqref{neq},
\vip
$\bullet$ 
\blu $\kappa^{(\mm)}_{d,p}$ by $d^{p/\mm}\kappa_{d,p}^{(\infty)}$ and $\theta^{(\mm)}_{d,p,q}$ by $\theta^{(\infty)}_{d,p,q}$ \bla
in items (i) and (iii);
\vip
$\bullet$
\blu $\kappa^{(\mm)}_{d,p,N}$ by $d^{p/\mm}\kappa_{d,p,N}^{(\infty)}$ and $\theta^{(\mm)}_{d,p,q,N}$ \bla by 
$\theta^{(\infty)}_{d,p,q,N}$
in item (ii).

\vip

\blu In each case, we present the bound under the form
$$
\hbox{{\tt \Big(diameter of $B_\mm(0,1)\Big)^p$}$\!\!\times${\tt 
\Big(bound in the compact case\Big)}$\times [\cM_q^{(\mm)}(\mu)]^{p/q}\times${\tt
\Big($\theta_{d,p,q}^{(\mm)}$ or $\theta_{d,p,q,N}^{(\mm)}$\Big)},}
$$
where {\it diameter of $B_\mm(0,1)$} equals $2$ and where by {\it compact case} 
we mean the case where $\mu$ is supported by the ball $B_\mm(0,1/2)$ with diameter $1$. \bla
\vip
One can tediously check that in each case, \blu $q \mapsto \theta_{d,p,q}^{(\mm)}$
(or $q \mapsto \theta_{d,p,q,N}^{(\mm)}$) \bla is decreasing
and tends to $1$ as $q\to \infty$. Hence if we are in the compact case, i.e. if $\mu$ is supported in 
\blu $B_\mm(0,1/2)$, \bla then 
\blu $\cM_q^{(\mm)}(\mu)\leq 2^{-q}$ \bla for all $q>0$, and we find that 
$$
\blu \lim_{q\to \infty} \hbox{{\tt \Big(diameter of $B_\mm(0,1)\Big)^p$}$\times[\cM_q^{(\mm)}(\mu)]^{p/q}\times${\tt
\Big($\theta_{d,p,q}^{(\mm)}$ or $\theta_{d,p,q,N}^{(\mm)}$\Big)}}=1, \bla
$$
which justifies the denomination {\tt bound in the compact case}.

\vip

One can tediously check
from (iii) that \blu $p \mapsto (\kappa_{d,p}^{(\mm)})^{1/p}$ \bla is increasing for $p\in (0,d/2)$, which is natural
by monotony in $p$ of $[\cT_p(\mu,\nu)]^{1/p}$. Actually, \blu when $\mm \in [1,\infty)$, it holds that
$p \mapsto \bla (\kappa_{d,p,r}^{(\mm)})^{1/p}$ \bla is increasing for $p\in (0,d/2)$ for each $r\geq 2$.

\vip

However, in the non compact case, we did not manage to guarantee such a property: it does not hold true that
in item (iii), \blu $p \mapsto [\kappa_{d,p}^{(\mm)}\theta_{d,p,q}^{(\mm)}]^{1/p}$ \bla is increasing 
for $p \in (0,d/2)$, as it should. Hence it may be sometimes be preferable to use the bound:
for $p\in (0,d/2)$ and $q>dp/(d-p)$,
\begin{equation}\label{grosdefaut}
\blu \E[\cT_p^{(\mm)}(\mu_N,\mu)]\!\leq \!\inf_{p'\in [p,d/2)}\E[\cT_{p'}^{(\mm)}(\mu_N,\mu)]^{p/p'}\!\leq\!
\inf_{p'\in [p,d/2)}2^p \frac{[\kappa_{d,p'}^{(\mm)}]^{p/p'}}{N^{p/d}} [\cM_q^{(\mm)}(\mu)]^{p/q}[\theta_{d,p',q}^{(\mm)}]^{p/p'},
\bla
\end{equation}
with the convention that \blu $\theta_{d,p',q}^{(\mm)}=\infty$ \bla if $q\leq d p'/(d-p')$.
This is the major default of this work. We identified some computations that
might be done more carefully, but this led to awful complications, without producing some marked improvements.
\vip
Finally, one can easily check from (iii)
that for any $p>0$,
$$
\lim_{d\to\infty} \kappa_{d,p}^{(\infty)} = 2^p \qquad \hbox{ and } \qquad 
\blu \lim_{d\to\infty} \kappa_{d,p}^{(2)} = 4^p. \bla
$$ 
Concerning the Euclidean case, we use \eqref{KdVG} below which implies that $\lim_{d\to \infty} (K_d^{(2)})^{1/d}= 1$,
we find that \blu $\lim_{d\to\infty} \kappa_{d,p,r}^{(2)} = (r-1)^{-p}r^{2p}$ for each $r\geq 2$, and this (optimally)
equals $4^p$ with $r=2$. When $\mm \in [1,\infty)\setminus\{2\}$, we deduce from \eqref{Kdc} that
$$
\limsup_{d\to\infty} \kappa_{d,p}^{(\mm)} \leq 12^p.
$$
\bla

\subsection{Numerical values in the compact case}
We start with the maximum norm. 
\vip
\begin{center}
\noindent\fbox{\begin{minipage}{0.97\textwidth}
\small{
\begin{center}
\renewcommand{\arraystretch}{1.9}
\begin{tabular}{|c|c|c|c|c|c|c|c|c|c|c|}
\hline
$d=1$&$d=2$&$d=3$&$d=4$&$d=5$&$d=6$&$d=7$&$d=8$&$d=9$&$d=100$&$d=500$
\\\hline
\large{$\frac{2.42}{N^{1/2}}$}
&\large{$\frac{0.73\log N + 1}{N^{1/2}}$}
&\large{$\frac{3.72}{N^{1/3}}$}
&\large{$\frac{2.45}{N^{1/4}}$}
&\large{$\frac{2.09}{N^{1/5}}$}
&\large{$\frac{1.94}{N^{1/6}}$}
&\large{$\frac{1.87}{N^{1/7}}$}
&\large{$\frac{1.84}{N^{1/8}}$}
&\large{$\frac{1.82}{N^{1/9}}$}
&\large{$\frac{1.98}{N^{1/100}}$}
&\large{$\frac{2.00}{N^{1/500}}$}
\\\hline
\end{tabular}
\end{center}
{\sc Table 1.} Bound of $\E[\cT_1^{(\infty)}(\mu_N,\mu)]$ for $N\geq 1$ (actually $N\geq 4$ when $d=2$),
if $\mu\in \cP(B_\infty(0,1/2))$.}
\end{minipage}}
\end{center}
\vip
\begin{center}
\noindent\fbox{\begin{minipage}{0.97\textwidth}
\small{
\begin{center}
\renewcommand{\arraystretch}{1.9}
\begin{tabular}{|c|c|c|c|c|c|c|c|c|c|c|}
\hline
$d=1$&$d=2$&$d=3$&$d=4$&$d=5$&$d=6$&$d=7$&$d=8$&$d=9$&$d=100$&$d=500$
\\\hline
\large{$\frac{1.05}{N^{1/4}}$}
&\large{$\frac{1.42}{N^{1/4}}$}
&\large{$\frac{2.20}{N^{1/4}}$}
&\large{$\frac{\sqrt{0.73\log N+1.26}}{N^{1/4}}$}
&\large{$\frac{2.75}{N^{1/5}}$}
&\large{$\frac{2.20}{N^{1/6}}$}
&\large{$\frac{2.01}{N^{1/7}}$}
&\large{$\frac{1.92}{N^{1/8}}$}
&\large{$\frac{1.87}{N^{1/9}}$}
&\large{$\frac{1.98}{N^{1/100}}$}
&\large{$\frac{2.00}{N^{1/500}}$}
\\\hline
\end{tabular}
\end{center}
{\sc Table 2.} Bound of $\sqrt{\E[\cT_2^{(\infty)}(\mu_N,\mu)]}$ for $N\geq 1$ (actually $N\geq 8$ when $d=4$),
if $\mu\in \cP(B_\infty(0,1/2))$.}
\end{minipage}}
\end{center}
\vip
\blu
For all $\mm \in [1,\infty)$, we have the classical easy estimate
\begin{equation}\label{Kdc}
K_d^{(\mm)} \leq 3^d, \qquad d\geq 1.
\end{equation}
Concerning the Euclidean norm, when $d$ is large enough, some much better results are available.
By
Verger-Gaugry \cite[(1.1)-(1.3)-(1.4)]{vg}, where (1.1) is due to Rogers \cite{r}
(we know from the author that there is a typo in
\cite{vg} and $d\geq 8$ is the correct condition, instead of $d \geq 2$),
\begin{gather}\label{KdVG}
K_d^{(2)} \leq \max\{K_{d,1}^{(2)},K_{d,2}^{(2)},K_{d,3}^{(2)}\}, \blu \qquad d \geq 8, \bla
\end{gather}
where (recall \eqref{Nr}-\eqref{Kd};
we have  $N_r=\nu_{T,n}$ with $2T=1/r$ and $n=d$ in the notation of \cite{vg})
\begin{gather*}
K_{d,1}^{(2)}= d^2[\log d + \log \log d +5],\\
K_{d,2}^{(2)}=\frac{7^{4(\log 7)/7}}4\sqrt{\frac\pi 2}d^{3/2}\frac{2(d-1)\log d +\frac12\log d
+ \log(\frac{\pi\sqrt{2d}}{\sqrt{\pi d}-2})}
{(1-2/\log d)(1-2/\sqrt{\pi d})(\log d)^2},\\
K_{d,3}^{(2)}= \sqrt{2 \pi d} \frac{(d-1)\log(2d)+(d-1)\log\log d +\frac12\log d
+ \log(\frac{\pi\sqrt{2d}}{\sqrt{\pi d}-2})}
{(1-2/\log d)(1-2/\sqrt{\pi d})}. 
\end{gather*}
\vip 
\blu
\begin{center}
\noindent\fbox{\begin{minipage}{0.97\textwidth}
\small{
\begin{center}
\renewcommand{\arraystretch}{1.9}
\begin{tabular}{|c|c|c|c|c|c|c|c|c|c|c|c|}
\hline
$d=8$&$d=10$&$d=12$&$d=15$&$d=20$&$d=25$&$d=35$&$d=50$&$d=75$&$d=100$&$d=500$
\\\hline
\large{$\!\frac{12.47}{N^{1/8}}\!$}
&\large{$\!\frac{8.91}{N^{1/10}}\!$}
&\large{$\!\frac{7.67}{N^{1/12}}\!$}
&\large{\boldmath$\!\frac{6.74}{N^{1/15}}\!$}
&\large{\boldmath$\!\frac{6.00}{N^{1/20}}\!$}
&\large{\boldmath$\!\frac{5.60}{N^{1/25}}\!$}
&\large{\boldmath$\!\frac{5.17}{N^{1/35}}\!$}
&\large{\boldmath$\!\frac{4.85}{N^{1/50}}\!$}
&\large{\boldmath$\!\frac{4.60}{N^{1/75}}\!$}
&\large{\boldmath$\!\frac{4.47}{N^{1/100}}\!$}
&\large{\boldmath$\!\frac{4.12}{N^{1/500}}\!$}
\\\hline
\large{\boldmath$\!\frac{5.18}{N^{1/8}}\!$}
&\large{\boldmath$\!\frac{5.73}{N^{1/10}}\!$}
&\large{\boldmath$\!\frac{6.29}{N^{1/12}}\!$}
&\large{$\!\frac{7.11}{N^{1/15}}\!$}
&\large{$\!\frac{8.36}{N^{1/20}}\!$}
&\large{$\!\frac{9.47}{N^{1/25}}\!$}
&\large{$\!\frac{11.38}{N^{1/35}}\!$}
&\large{$\!\frac{13.76}{N^{1/50}}\!$}
&\large{$\!\frac{17.01}{N^{1/75}}\!$}
&\large{$\!\frac{19.73}{N^{1/100}}\!$}
&\large{$\!\frac{44.60}{N^{1/500}}\!$}
\\\hline
\end{tabular}
\end{center}
{\sc Table 3.} Bound of $\E[\cT_1^{(2)}(\mu_N,\mu)]$ for $N\geq 1$
if $\mu\in \cP(B_2(0,1/2))$, using the bound proposed for the Euclidean norm in 
Theorem \ref{mr}-(iii) (second line)
and using $\sqrt{d}$ times the bound proposed for the maximum norm (third line). In bold the one to be used.}
\end{minipage}}
\end{center}
\vip
\begin{center}
\noindent\fbox{\begin{minipage}{0.97\textwidth}
\small{
\begin{center}
\renewcommand{\arraystretch}{1.9}
\begin{tabular}{|c|c|c|c|c|c|c|c|c|c|c|}
\hline
$d=8$&$d=10$&$d=12$&$d=15$&$d=20$&$d=25$&$d=35$&$d=50$&$d=75$&$d=100$&$d=500$
\\\hline
\large{$\!\frac{12.95}{N^{1/8}}\!$}
&\large{$\!\frac{9.06}{N^{1/10}}\!$}
&\large{$\!\frac{7.73}{N^{1/12}}\!$}
&\large{\boldmath$\!\frac{6.76}{N^{1/15}}\!$}
&\large{\boldmath$\!\frac{6.00}{N^{1/20}}\!$}
&\large{\boldmath$\!\frac{5.60}{N^{1/25}}\!$}
&\large{\boldmath$\!\frac{5.17}{N^{1/35}}\!$}
&\large{\boldmath$\!\frac{4.85}{N^{1/50}}\!$}
&\large{\boldmath$\!\frac{4.60}{N^{1/75}}\!$}
&\large{\boldmath$\!\frac{4.47}{N^{1/100}}\!$}
&\large{\boldmath$\!\frac{4.12}{N^{1/500}}\!$}
\\\hline
\large{\boldmath$\!\frac{5.41}{N^{1/8}}\!$}
&\large{\boldmath$\!\frac{5.84}{N^{1/10}}\!$}
&\large{\boldmath$\!\frac{6.35}{N^{1/12}}\!$}
&\large{$\!\frac{7.13}{N^{1/15}}\!$}
&\large{$\!\frac{8.36}{N^{1/20}}\!$}
&\large{$\!\frac{9.47}{N^{1/25}}\!$}
&\large{$\!\frac{11.38}{N^{1/35}}\!$}
&\large{$\!\frac{13.76}{N^{1/50}}\!$}
&\large{$\!\frac{17.01}{N^{1/75}}\!$}
&\large{$\!\frac{19.73}{N^{1/100}}\!$}
&\large{$\!\frac{44.60}{N^{1/500}}\!$}
\\\hline
\end{tabular}
\end{center}
{\sc Table 4.} Bound of $\sqrt{\E[\cT_2^{(2)}(\mu_N,\mu)]}$ for $N\geq 1$
if $\mu\in \cP(B_2(0,1/2))$, using the bound proposed for the Euclidean norm in 
Theorem \ref{mr}-(iii) (second line)
and using $\sqrt{d}$ times the bound proposed for the maximum norm (third line). In bold the one to be used.}
\end{minipage}}
\end{center}

\vip
As we can see, in large dimension, the bounds concerning $p=1$ and $p=2$ are very similar.
We also see that, for $p=1$ and $p=2$, it is better to use $\sqrt{d}$ 
times the bound proposed for the maximum norm
when $d\in \{8,\dots,12\}$. If using \eqref{Kdc}, we find that 
for $d\in \{2,\dots,7\}$, it is also better to use $\sqrt{d}$ 
times the bound proposed for the maximum norm.

\bla
\subsection{Numerical values in the non-compact case}

Here we study how \blu $\theta^{(\mm)}_{d,p,q}$ \bla is far from $1$.
When $p=d/2$, we observe that \blu $N \mapsto \theta^{(\mm)}_{d,p,q,N}$ \bla is decreasing, and we 
e.g. study \blu $\theta^{(\mm)}_{d,p,q,100}$ \bla (which controls \blu $\theta^{(\mm)}_{d,p,q,N}$ \bla 
for all $N\geq 100$).
We start with the maximum norm.
\vip

\begin{center}
\noindent\fbox{\begin{minipage}{0.95\textwidth}
\small{
\begin{center}
\renewcommand{\arraystretch}{1.3}
\begin{tabular}{|c|c|c|c|c|c|c|c|c|c|c|}
\hline
$d=1$  &$d=2$  &$d=3$  &$d=4$  &$d=5$  &$d=6$  &$d=7$  &$d=8$  &$d=9$  &$d=100$ &$d=500$
\\\hline
$4.4$  &$4.2$  &$3.3$  &$3.0$  &$2.9$  &$2.8$  &$2.8$  &$2.7$  &$2.7$  &$2.5$   &$2.4$
\\\hline
$9.8$  &$9.4$  &$7.3$  &$6.8$  &$6.5$  &$6.4$  &$6.3$  &$6.2$  &$6.1$  &$5.5$   &$5.5$
\\\hline
$40.4$ &$39.0$ &$29.9$ &$27.7$ &$26.6$ &$25.9$ &$25.3$ &$25.0$ &$24.6$ &$22.3$  &$22.1$
\\\hline
\end{tabular}
\end{center}
{\sc Table 5.} Here $p=1$.
Minimum value of $q$ so that $\theta_{d,1,q}^{(\infty)} \leq c$ if
$d\neq 2$ or $\theta_{d,1,q,100}^{(\infty)} \leq c$ if $d= 2$,
with $c=4$ (second line), $c=2$ (third line), $c=1.25$ (fourth line).}
\end{minipage}}
\end{center}
\vip
\begin{center}
\noindent\fbox{\begin{minipage}{0.95\textwidth}
\small{
\begin{center}
\renewcommand{\arraystretch}{1.3}
\begin{tabular}{|c|c|c|c|c|c|c|c|c|c|c|}
\hline
$d=1$  &$d=2$  &$d=3$  &$d=4$  &$d=5$  &$d=6$  &$d=7$  &$d=8$  &$d=9$  &$d=100$ &$d=500$
\\\hline
$5.1$  &$5.0$  &$4.9$  &$4.9$  &$4.1$  &$3.7$  &$3.5$  &$3.3$  &$3.2$  &$2.6$   &$2.5$
\\\hline
$9.5$  &$8.9$  &$8.4$  &$8.4$  &$6.9$  &$6.4$  &$6.0$  &$5.8$  &$5.7$  &$4.6$   &$4.5$
\\\hline
$37.0$ &$34.5$ &$32.4$ &$32.5$ &$26.7$ &$24.6$ &$23.4$ &$22.5$ &$21.9$ &$17.7$  &$17.4$
\\\hline
\end{tabular}
\end{center}
{\sc Table 6.} Here $p=2$.
Minimum value of $q$ so that $\sqrt{\theta_{d,2,q}^{(\infty)}} \leq c$ if
$d\neq 4$ or $\sqrt{\theta_{d,2,q,100}^{(\infty)}} \leq c$ if $d=4$,
with $c=4$ (second line), $c=2$ (third line), $c=1.25$ (fourth line).}
\end{minipage}}
\end{center}
\vip
Comparing Tables 5 and 6, it seems clear that, at least for large values
of $d$, it is preferable to use the bound \eqref{grosdefaut}.

\vip

We do the same job concerning the Euclidean norm. 
We only deal with $\theta_{d,1,q}^{(2)}$ as defined in Theorem \ref{mr} for simplicity,
even if we recall that it is preferable to use \eqref{neq} and the bound concerning the maximum norm
for low dimensions.
\vip

\blu
\begin{center}
\noindent\fbox{\begin{minipage}{0.95\textwidth}
\small{
\begin{center}
\renewcommand{\arraystretch}{1.3}
\begin{tabular}{|c|c|c|c|c|c|c|c|c|c|c|c|}
\hline
$\!d=8\!$ &$\!d=10\!$ &$\!d=12\!$ &$\!d=15\!$ &$\!d=20\!$ &$\!d=25\!$ &$\!d=35\!$ &$\!d=50\!$ &$\!d=75\!$ &$\!d=100\!$ &$\!d=500\!$
\\\hline
$2.4$     &$2.4$      &$2.3$      &$2.3$      &$2.3$      &$2.3$      &$2.3$      &$2.3$      &$2.3$      &$2.3$       &$2.3$
\\\hline
$5.3$     &$5.2$      &$5.2$      &$5.1$      &$5.1$      &$5.1$      &$5.0$      &$5.0$      &$5.0$      &$5.0$       &$5.0$
\\\hline
$21.8$    &$21.5$     &$21.3$     &$21.1$     &$20.9$     &$20.8$     &$20.7$     &$20.6$     &$20.6$     &$20.6$      &$20.5$
\\\hline
\end{tabular}
\end{center}
{\sc Table 7.} Here $p=1$.
Minimum value of $q$ so that $\theta_{d,1,q}^{(2)} \leq c$
with $c=4$ (second line), $c=2$ (third line), $c=1.25$ (fourth line).}
\end{minipage}}
\end{center}
\vip
\begin{center}
\noindent\fbox{\begin{minipage}{0.95\textwidth}
\small{
\begin{center}
\renewcommand{\arraystretch}{1.3}
\begin{tabular}{|c|c|c|c|c|c|c|c|c|c|c|c|}
\hline
$\!d=8\!$ &$\!d=10\!$ &$\!d=12\!$ &$\!d=15\!$ &$\!d=20\!$ &$\!d=25\!$ &$\!d=35\!$ &$\!d=50\!$ &$\!d=75\!$ &$\!d=100\!$ &$\!d=500\!$
\\\hline
$3.2$     &$3.0$      &$2.9$      &$2.8$      &$2.7$      &$2.6$      &$2.6$      &$2.5$      &$2.5$      &$2.5$      &$2.5$
\\\hline
$5.4$     &$5.1$      &$4.9$      &$4.7$      &$4.5$      &$4.5$      &$4.4$      &$4.3$      &$4.2$      &$4.2$      &$4.2$ 
\\\hline
$20.5$    &$19.3$     &$18.5$     &$17.9$     &$17.3$     &$17.0$     &$16.6$     &$16.4$     &$16.2$     &$16.1$     &$15.9$ 
\\\hline
\end{tabular}
\end{center}
{\sc Table 8.} Here $p=2$.
Minimum value of $q$ so that $\sqrt{\theta_{d,2,q}^{(2)}} \leq c$
with $c=4$ (second line), $c=2$ (third line), $c=1.25$ (fourth line).}
\end{minipage}}
\end{center}
\vip

Comparing Tables 7 and 8, it seems again clear that
it is not vain to use the bound \eqref{grosdefaut}.
\bla

\subsection{On a possible lowerbound}

As mentioned to us by Pag\`es, we have the following lowerbound, 
holding for any \blu $\mm\in [1,\infty)\cup\{\infty\}$. \bla
Consider $X_1,\dots,X_N$ independent and $\mu$-distributed.
It holds that \blu $\cT_p^{(\mm)}(\mu_N,\mu) \geq \cS_p^{(\mm)}(\mu;X_1,\dots,X_N)$, where
$$
\cS_p^{(\mm)}(\mu;x_1,\dots,x_d)=\inf\Big\{\cT_p^{(\mm)}\Big(\mu,\sum_{i=1}^N \alpha_i \delta_{x_i}\Big) 
: (\alpha_i)_{i=1,\dots,N} \in [0,1]^N,\;\sum_{i=1}^N \alpha_i=1\Big\},
$$
and Luschgy-Pag\`es show in \cite{lp} that, under some technical conditions on $\mu$,
$$
\lim_{N\to \infty} N^{-p/d} \E[\cS_p^{(\mm)}(\mu;X_1,\dots,X_n)]=\frac{\Gamma(1+p/d)}{[\lambda_d(B_{(\mm)}(0,1))]^{p/d}} 
\int_{\rd} [f(x)]^{1-p/d} \dd x,
$$
where $f$ is the density of $\mu$, where $\lambda_d$ is the Lebesgue measure on $\rd$
and where $\Gamma$ is the classical $\Gamma$ function.
Choosing for $\mu$ the uniform law on $B_\mm(0,1/2)$, we find that
$$
\liminf_{N \to \infty} N^{-p/d}\E[\cT_p^{(\mm)}(\mu_N,\mu)] 
\geq  \lim_{N\to \infty} N^{-p/d} \E[\cS_p^{(\mm)}(\mu;X_1,\dots,X_n)]
=\frac{\Gamma(1+p/d)}{2^p}=: \gamma_{d,p},
$$
to be compared with $\kappa_{d,p}^{(\mm)}$ when $p \in (0,d/2)$.
This may be a rough lowerbound, because this is an asymptotic bound as $N\to \infty$, and because
$\cS_p^{(\mm)}(\mu;X_1,\dots,X_N)$ is likely to be really smaller 
than $\cT_p^{(\mm)}(\mu_N,\mu)$ (in particular it decreases in $N^{-p/d}$ when $p>d/2$ while $\cT_p^{(\mm)}(\mu_N,\mu)$
decreases in $N^{-1/2}$ in general when $p>d/2$).\bla

\vip

We have $\gamma_{d,1}\in (0.44,0.5)$ for all $d \geq 3$, to be compared with
the numerators in Tables 1 and 3. Hence when $p=1$ and say $d\geq 4$, see Table 1,
$\kappa_{d,1}^{(\infty)}$ is at worst \blu $(2.45/0.44)\simeq 5.6$ \bla times too large.
When $p=1$ and $d\geq 8$, see Table 3, $\min\{\kappa_{d,1}^{(2)},\sqrt d \kappa_{d,1}^{(\infty)}\}$
is at worst \blu $(6.74/0.44)\simeq 15.4$ \bla times too large.
We hope this is pessimistic.
\vip
It holds that $\sqrt{\gamma_{d,2}}\in (0.47,0.5)$ for all $d \geq 5$, 
to be compared with the numerators in Tables 2 and 4. When $p=2$ and $d\geq 5$, see Table 2,
$(\kappa_{d,2}^{(\infty)})^{1/2}$ is at worst \blu $(2.75/0.47)\simeq 5.9$ \bla times too large. 
When $p=2$ and $d\geq 8$, see Table 4, $\min\{(\kappa_{d,2}^{(2)})^{1/2},\sqrt d (\kappa_{d,2}^{(\infty)})^{1/2}\}$
is at worst \blu $(6.76/0.47)\simeq 14.4$ \bla times too large.
Again, we hope this is pessimistic.

\vip

We do not discuss the non compact case, but the numerical results do not seem
quite favorable.

\subsection{The case with a low order finite moment}
Since this last result is likely to be much
less useful for applications, we only treat the case of the maximum norm.

\begin{thm}\label{mrn}
Let $q>p>0$ such that $q<\min\{2p,dp/(d-p)\}$, i.e. $q\in(p,2p)$ if 
$p\geq d/2$ and $q\in(p,dp/(d-p))$ if $p\in (0,d/2]$. 
Fix $\mu \in \cP(\rd)$.
For all $N\geq 1$,
$$
\E[\cT_p^{(\infty)}(\mu_N,\mu)]\leq 2^p [\cM_q^{(\infty)}(\mu)]^{p/q} \frac{\zeta_{d,p,q}^{(\infty)}}{N^{(q-p)/q}},
$$
where, setting $\e_p=\max\{2^{-1},2^{-p}\}$,
$$
\zeta_{d,p,q}^{(\infty)}=\Big[\e_p 2^{2p/q-1} +\Big(\frac{2^{d-1}}{2^{d/2}-1}\Big)^{2(q-p)/q}\frac{1-2^{-p}}{1-2^{d-p-dp/q}}\Big]
\min_{a\in(1,\infty)} \Big(\frac{a^p}{a^{p-q/2}-1}+ \frac{a^p}{1-a^{p-q}}\Big).
$$
\end{thm}

\subsection{Plan of the paper}
In Section \ref{gu}, we provide a general estimate of the transport cost between two measures.
In Section \ref{gume}, we apply this general estimate to derive a bound of
$\E[\cT_p(\mu_N,\mu)]$, for a general norm, for all the values of $p>0$ and in any dimension.
In Section \ref{series}, we precisely study some elementary series.
We obtain a bound of
$\E[\cT_p(\mu_N,\mu)]$, for a general norm, separating the cases $p>d/2$, $p=d/2$ and $p\in (0,d/2)$
in Section \ref{titrenaze}.
We conclude the proof of Theorem \ref{mr} for the maximum norm in Section \ref{concmax}
and for the \blu other norms \bla in Section \ref{conceuc}.
Finally, we check Theorem \ref{mrn} in Section \ref{lom}.

\section{Upperbound of the transport cost between two measures}\label{gu}

The result we prove in this section, Proposition \ref{depart}, 
is more or less classical, see  Boissard-Le Gouic \cite[Proposition 1.1]{blg},
Dereich-Scheutzow-Schottstedt \cite[Lemma 2 and Theorem 3]{dss},
Fournier-Guillin \cite[Lemma 5]{fg} and Weed-Bach \cite[Proposition 1]{wb}.
As noted in \cite{wb}, similar ideas can already be found in Ajtai-Koml\'os-Tusn\'ady \cite{akt}.
However, we provide a slightly more
precise version, that allows us to get some smaller constants.
We consider fixed the following objects.

\begin{set}\label{sss}
(a) We fix a norm $|\cdot|$ on $\rr^d$. We denote by $B(x,r)=\{y\in\rd : |x-y|<r\}$ the corresponding balls,
by $\cT_p(\mu,\nu)=\inf_{\xi \in \cH(\mu,\nu)} \int_{\rd\times\rd}|x-y|^p \xi(\dd x,\dd y)$
the corresponding transport cost, and by $\cM_q(\mu)=\int_{\rd}|x|^q \mu(\dd x)$ the corresponding moments.
\vip
(b) Let $G_0=B(0,1)$. For $a>1$, we set $G_0^a=G_0$ and, for all $n\geq 1$, 
$G_n^a=B(0,a^n)\setminus B(0,a^{n-1})$.
\vip
(c) We consider a family $(\cQ_\ell)_{\ell=0,\dots,k}$ of nested partitions of $G_0$ such that $\cQ_0=\{G_0\}$.
For each $\ell=1,\dots,k$, each $C \in \cQ_{\ell}$,
there exists a unique $F \in \cQ_{\ell-1}$ such that $C \subset F$; we then say that $C$ is a child of $F$. 
For $\ell=0,\dots,k$, we denote by $|\cQ_\ell|$ the cardinal of 
$\cQ_\ell$ and we set
$$
\delta_\ell = \max_{C \in \cQ_\ell} \sup_{x,y \in C} |x-y|.
$$
\end{set}

For $C\subset \rr^d$ and $r>0$, we put $rC=\{rx : x \in C\}$. Recall that $\e_p=2^{-1}\lor 2^{-p}$.

\begin{prop}\label{depart}
We adopt Setting \ref{sss} and consider
$\mu,\nu \in \cP(\rr^d)$. 
For all $a>1$,
there are some nonnegative numbers 
$(r_{a,n,\ell}(\mu,\nu))_{n\geq 0,\ell= 0,\dots,k-1}$ satisfying 
\begin{gather}
\sum_{\ell=0}^{k-1} r_{a,n,\ell}(\mu,\nu)\leq 1\label{rnl1}
\end{gather}
for all $n\geq 0$ and, with the convention that $0/0=0$,
\begin{gather}
r_{a,n,\ell}(\mu,\nu)\!\leq\! \frac12\sum_{F\in\cQ_\ell} 
\Big(\frac{\mu(a^nF \cap G_n^a)}{\mu(G_n^a)}\land\frac{\nu(a^nF \cap G_n^a)}{\nu(G_n^a)}\Big)\hskip-5pt
\sum_{C \!\hbox{\tiny{ child of }}\! F}
\Big|\frac{\mu(a^nC \cap G_n^a)}{\mu(a^nF \cap G_n^a)}-\frac{\nu(a^nC \cap G_n^a)}{\nu(a^nF \cap G_n^a)} \Big|
\label{rnl2}
\end{gather}
for all $n\geq 0$, all $\ell=0,\dots,k-1$, 
and such that for all $p>0$,
\begin{align*}
\cT_p(\mu,\nu) \leq& \sum_{n\geq 0} a^{pn}\Big( 2^p\e_p |\mu(G_n^a)-\nu(G_n^a)| + 
(\mu(G_n^a)\land\nu(G_n^a)) \Big[\delta_k^p+\sum_{\ell= 0}^{k-1} \delta_\ell^p r_{a,n,\ell}(\mu,\nu)\Big] \Big).
\end{align*}
\end{prop}

The coefficients $r_{a,n,\ell}(\mu,\nu)$ are  actually explicit, but it seems difficult
to use more than the properties \eqref{rnl1}-\eqref{rnl2}.
We start with the compact case.

\begin{lem}\label{cc}
Let $\mu,\nu \in \cP(B(0,1))$. There is
$(u_{\ell}(\mu,\nu))_{\ell=0,\dots,k-1} \in \rr_+^{k}$ satisfying
\begin{gather}
\sum_{\ell=0}^{k-1} u_\ell(\mu,\nu)\leq 1 \label{sul}
\end{gather}
and
\begin{gather}
u_{\ell}(\mu,\nu)\leq \frac12\sum_{F\in\cQ_\ell} (\mu(F)\land \nu(F)) \hskip-5pt\sum_{C \!\hbox{\tiny{ child of }}\! F}
\Big|\frac{\mu(C)}{\mu(F)}-\frac{\nu(C)}{\nu(F)} \Big|, \qquad \ell=0,\dots,k-1, \label{ul}
\end{gather}
and such that for all $p>0$,
\begin{align*}
\cT_p(\mu,\nu) \leq& \delta_k^p+ \sum_{\ell=0}^{k-1} \delta_\ell^{p} u_{\ell}(\mu,\nu).
\end{align*}
\end{lem}

\begin{proof}
For all $0\leq i\leq\ell\leq k$ and $C\in \cQ_\ell$, let
$f_i(C)$ be the unique element of $\cQ_i$ containing $C$.

\vip

{\it Step 1: construction of the coupling.} For all $F \in \cQ_k$, we set
\begin{equation}\label{ttc}
\xi_{F}(\dd x,\dd y)=\frac{\mu|_{F}(\dd x)}{\mu(F)}\frac{\nu|_{F}(\dd y)}{\nu(F)}.
\end{equation}
Then by reverse induction, for $\ell\in\{0,\dots,k-1\}$ and $F\in \cQ_\ell$, we build
\begin{align}\label{dfck}
\xi_{F}(\dd x,\dd y)=\sum_{C \!\hbox{\tiny{ child of }}\! F}
\rho_C \xi_C(\dd x,\dd y) + q_F \alpha_F(\dd x) \beta_F(\dd y),
\end{align}
where
$$
\rho_C= \frac{\mu(C)}{\mu(F)}\land \frac{\nu(C)}{\nu(F)},
$$
which depends only on $C \in \cQ_{\ell+1}$ since $F=f_{\ell}(C)$, where
\begin{equation}\label{ano0}
q_F=\frac12\sum_{C \!\hbox{\tiny{ child of }}\! F}
\Big|\frac{\mu(C)}{\mu(F)}-\frac{\nu(C)}{\nu(F)} \Big|,
\end{equation}
and where
\begin{align*}
\alpha_F(\dd x)=& \frac1{q_F} \sum_{C \!\hbox{\tiny{ child of }}\! F} 
\Big(\frac{\mu(C)}{\mu(F)}-\frac{\nu(C)}{\nu(F)} \Big)_+\frac{\mu|_{C}(\dd x)}{\mu(C)},\\
\beta_F(\dd x)=& \frac1{q_F} \sum_{C \!\hbox{\tiny{ child of }}\! F} 
\Big(\frac{\nu(C)}{\nu(F)}-\frac{\mu(C)}{\mu(F)} \Big)_+\frac{\nu|_{C}(\dd x)}{\nu(C)}.
\end{align*}
It holds that $\alpha_F$ and $\beta_F$ are two probability measures on $F$, because
\begin{equation}\label{ano}
q_F=\sum_{C \!\hbox{\tiny{ child of }}\! F}
\Big(\frac{\mu(C)}{\mu(F)}-\frac{\nu(C)}{\nu(F)} \Big)_+
=\sum_{C \!\hbox{\tiny{ child of }}\! F}
\Big(\frac{\nu(C)}{\nu(F)}-\frac{\mu(C)}{\mu(F)} \Big)_+.
\end{equation}

{\it Step 2.} Here we show that $\xi_{G_0} \in \cH(\mu,\nu)$,
so that $\cT_p(\mu,\nu)\leq \int_{\rd\times\rd}|x-y|^p\xi_{G_0}(\dd x,\dd y)$.
We recall that $G_0=B(0,1)$ is the unique element of $\cQ_0$.

\vip
We actually prove by 
reverse induction that for all $\ell\in\{0,\dots,k\}$, all $F \in \cQ_\ell$, it holds that 
$\xi_F \in \cH(\frac{\mu|_{F}}{\mu(F)}, \frac{\nu|_{F}}{\nu(F)})$.
The result will then follow by choosing $\ell=0$ and $F=G_0$.
\vip
This is obvious if $\ell=k$, see \eqref{ttc}. 
Next, we assume that this holds true for some $\ell+1 \in \{1,\dots,k\}$,
and we consider $F\in \cQ_\ell$.
For $A\in\cB(\rr^d)$, we use \eqref{dfck} to write
$$
\xi_F(A\times \rr^d)=\sum_{C \!\hbox{\tiny{ child of }}\! F}\rho_C \xi_{C}(A\times \rr^d) + q_F\alpha_F(A)
=\sum_{C \!\hbox{\tiny{ child of }}\! F}\rho_C \frac{\mu(A\cap C)}{\mu(C)} + q_F\alpha_F(A)
$$
by induction assumption. Thus
$$
\xi_F(A\times \rr^d)=\hskip-5pt\sum_{C \!\hbox{\tiny{ child of }}\! F} \Big[\frac{\mu(C)}{\mu(F)}\land \frac{\nu(C)}{\nu(F)}
+\Big(\frac{\mu(C)}{\mu(F)}-\frac{\nu(C)}{\nu(F)} \Big)_+ \Big]\frac{\mu(A\cap C)}{\mu(C)}
=\hskip-5pt\sum_{C \!\hbox{\tiny{ child of }}\! F} \frac{\mu(C)}{\mu(F)}\frac{\mu(A\cap C)}{\mu(C)},
$$
whence $\xi_F(A\times \rr^d)=\mu(A\cap F)/\mu(F)$. One shows 
similarly that $\xi_F(\rr^d\times A)=\nu(A\cap F)/\nu(F)$.

\vip

{\it Step 3.} For $i\in\{0,\dots,k\}$ and $F\in \cQ_i$, we put $m_F=\int_{\rd\times\rd}|x-y|^p\xi_F(\dd x,\dd y)$.
In this step, we show by induction that for all $i\in\{0,\dots,k-1\}$,
\begin{equation}\label{rectd}
m_{G_0}\leq \delta_0^pq_{G_0}+\sum_{\ell=1}^i \delta_{\ell}^p\sum_{C_\ell\in \cQ_\ell} \Big(\prod_{j=1}^\ell\rho_{f_j(C_\ell)}\Big)
q_{C_\ell}+ \sum_{C_{i+1}\in\cQ_{i+1}} \Big(\prod_{j=1}^{i+1}\rho_{f_j(C_{i+1})}\Big)m_{C_{i+1}}.
\end{equation}

Recalling \eqref{dfck}, since the set of all the children of $G_0$ is $\cQ_1$, and since
$|x-y|\leq \delta_0$ for all $x,y\in G_0$, so that $\int_{\rd\times\rd}|x-y|^p\alpha_{G_0}(\dd x)\beta_{G_0}(\dd y)
\leq \delta_0^p$, we see that
$$
m_{G_0} \leq \delta_0^p q_{G_0} + \sum_{C_1 \in \cQ_1} \rho_{C_1} m_{C_1}
= \delta_0^p q_{G_0} + \sum_{C_1 \in \cQ_1} \rho_{f_1(C_1)} m_{C_1},
$$
which is \eqref{rectd} with $i=0$. Assume now that \eqref{rectd} holds true for some 
$i \in \{0,\dots,k-2\}$. For all $C_{i+1}\in\cQ_{i+1}$, we use \eqref{dfck} and 
that $\int_{\rd\times\rd}|x-y|^p\alpha_{C_{i+1}}(\dd x)\beta_{C_{i+1}}(\dd y) \leq \delta_{i+1}^p$
to write
$$
m_{C_{i+1}} \leq \delta_{i+1}^p q_{C_{i+1}} + \sum_{C_{i+2} \!\hbox{\tiny{ child of }}\! C_{i+1}} 
\rho_{C_{i+2}}m_{C_{i+2}}.
$$
Hence, since $f_j(C_{i+2})=f_j(C_{i+1})$
for all $j=1,\dots,i+1$ if $C_{i+2}$ is a child of $C_{i+1}$,
\begin{align*}
\sum_{C_{i+1}\in\cQ_{i+1}} \hskip-5pt \Big(\prod_{j=1}^{i+1}\rho_{f_j(C_{i+1})}\Big)m_{C_{i+1}} \leq &
 \delta_{i+1}^p \hskip-5pt
\sum_{C_{i+1}\in\cQ_{i+1}} \hskip-5pt\Big(\prod_{j=1}^{i+1}\rho_{f_j(C_{i+1})}\Big)q_{C_{i+1}}
+ \hskip-5pt\sum_{C_{i+2}\in\cQ_{i+2}} \hskip-5pt\Big(\prod_{j=1}^{i+2}\rho_{f_j(C_{i+2})}\Big)m_{C_{i+2}}.
\end{align*}
This last formula, inserted in \eqref{rectd}, gives \eqref{rectd} with $i+1$ instead of $i$.

\vip

{\it Step 4.} For all $C_{k} \in \cQ_k$, we have $m_{C_k} \leq \delta_k^p$ by \eqref{ttc} and since 
$x,y\in C_k$ implies
that $|x-y| \leq \delta_k$. Hence, by definition of $\rho_F$,
$$
\sum_{C_{k}\in\cQ_{k}} \Big(\prod_{j=1}^{k}\rho_{f_j(C_{k})}\Big)m_{C_{k}}
\leq \delta_k^p \sum_{C_{k}\in\cQ_{k}} \prod_{j=1}^{k}\frac{\mu(f_j(C_{k}))}{\mu(f_{j-1}(C_{k}))}
=\delta_k^p \sum_{C_{k}\in\cQ_{k}}\mu(C_{k})=\delta_k^p.
$$
This, inserted in \eqref{rectd} with $i=k-1$, tells us that
$$
m_{G_0}\leq \delta_0^p q_{G_0} +\sum_{\ell=1}^{k-1} \delta_\ell^{p}\sum_{C_\ell\in \cQ_\ell} \Big(\prod_{j=1}^\ell\rho_{f_j(C_\ell)}
\Big)q_{C_\ell}+ \delta_k^p.
$$
Since $\cT_p(\mu,\nu)\leq m_{G_0}$ by Step 2, we conclude that
$$
\cT_p(\mu,\nu) \leq  \sum_{\ell=0}^{k-1} \delta_\ell^{p}u_\ell(\mu,\nu) + \delta_k^p,
$$
where 
$$
u_0(\mu,\nu)=q_{G_0} \qquad \hbox{and} \qquad 
u_\ell(\mu,\nu)=\sum_{C_\ell\in \cQ_\ell} (\prod_{j=1}^\ell\rho_{f_j(C_\ell)})q_{C_\ell} 
\quad\hbox{for}\quad \ell\in \{1,\dots,k-1\}.
$$

{\it Step 5.} We now check by induction that for all $n=1,\dots,k$,  
\begin{align}\label{find}
\sum_{\ell=0}^{n-1} u_\ell(\mu,\nu)=1- \sum_{C_n \in \cQ_n} \prod_{j=1}^n\rho_{f_j(C_n)},
\end{align}
and this will imply \eqref{sul}. If first $n=1$, by \eqref{ano},
$$
u_0(\mu,\nu)=q_{G_0}=\sum_{C_1\in\cQ_1}(\mu(C_1)-\nu(C_1))_+=1-\sum_{C_1\in\cQ_1}[\mu(C_1)\land\nu(C_1)]
=1-\sum_{C_1\in\cQ_1} \rho_{C_1}
$$
as desired. If next \eqref{find} holds with some $n\in \{1,k-1\}$, we write
$$
\sum_{\ell=0}^{n} u_\ell(\mu,\nu)=1- \sum_{C_n \in \cQ_n} \prod_{j=1}^n\rho_{f_j(C_n)}+
\sum_{C_n\in \cQ_n} \Big(\prod_{j=1}^n\rho_{f_j(C_n)}\Big)q_{C_n}
=1- \sum_{C_{n+1} \in \cQ_{n+1}} \prod_{j=1}^{n+1}\rho_{f_j(C_{n+1})},
$$
because, recalling \eqref{ano} and 
that  $\rho_{C_{n+1}}=\frac{\mu(C_{n+1})}{\mu(C_n)}\land\frac{\nu(C_{n+1})}{\nu(C_n)}$
(if $C_{n+1}$ is a child of $C_n$),
\begin{align*}
q_{C_n}=&\sum_{C_{n+1} \!\hbox{\tiny{ child of }}\! C_n}
\Big(\frac{\mu(C_{n+1})}{\mu(C_n)}-\frac{\nu(C_{n+1})}{\nu(C_n)} \Big)_+
=1 - \sum_{C_{n+1} \!\hbox{\tiny{ child of }}\! C_n} \rho_{C_{n+1}}.
\end{align*}

\vip

{\it Step 6.} It only remains to verify \eqref{ul}. But for $\ell=1,\dots,k-1$, by definition \eqref{ano0}
of $q_{C_\ell}$ and 
since
$\prod_{j=1}^\ell\rho_{f_j(C_\ell)} \leq\mu(C_\ell)\land\nu(C_\ell)$,
$$
u_{\ell}(\mu,\nu)=\sum_{C_\ell\in \cQ_\ell} \Big(\prod_{j=1}^\ell\rho_{f_j(C_\ell)}\Big)q_{C_\ell}\leq 
\frac12 \sum_{C_\ell\in \cQ_\ell} (\mu(C_\ell)\land\nu(C_\ell)) \sum_{C_{\ell+1} \!\hbox{\tiny{ child of }}\! C_\ell}
\Big|\frac{\mu(C_{\ell+1})}{\mu(C_\ell)}-\frac{\nu(C_{\ell+1})}{\nu(C_\ell)} \Big|.
$$
Hence we have \eqref{ul} for any $\ell=1,\dots,k-1$. Next, since $\cQ_0=\{G_0\}$
and $\mu,\nu$ are carried by $G_0$,
$$
u_0(\mu,\nu)=q_{C_0}=\frac12 \sum_{C_1\in \cQ_1}|\mu(C_1)-\nu(C_1)|=\frac12\sum_{C_0\in\cQ_0} (\mu(C_0)\land\nu(C_0))
\sum_{C_1\in \cQ_1}\Big|\frac{\mu(C_1)}{\mu(C_0)}-\frac{\nu(C_1)}{\nu(C_0)}\Big|,
$$
whence \eqref{ul} with $\ell=0$.
\end{proof}

We next consider the non compact case.

\begin{lem}\label{l2}
For any $\mu,\nu \in \cP(\rr^d)$, any $a>1$, any $p>0$,
\begin{align}\label{ccc}
\cT_p(\mu,\nu)\leq \sum_{n\geq 0}a^{pn}\Big(2^p \e_p|\mu(G_n^a)-\nu(G_n^a)|
+(\mu(G_n^a)\land \nu(G_n^a))\cT_p(\cR_n^a\mu,\cR_n^a\nu)\Big),
\end{align}
where $\cR_n^a\mu$ is the image measure of $\frac{\mu|_{G_n^a}}{\mu(G_n^a)}$ by the map
$x \mapsto a^{-n}x$.
\end{lem}

\begin{proof}
We fix $a>1$ and $p>0$ and consider, for each $n\geq 0$, the optimal coupling $\pi_n$
between $\cR_{n}^a\mu$ and $\cR_{n}^a\nu$ for $\cT_p$. We define $\xi_n$
as the image of $\pi_n$ by the map $(x,y)\mapsto (a^nx,a^ny)$. It holds that $\xi_n$ belongs
to $\cH(\mu\vert_{G_n^a}/\mu(G_n^a),\nu\vert_{G_n^a}/\nu(G_n^a))$
and satisfies 
\begin{align}\label{jab2}
\int_{\rr^d\times\rr^d} |x-y|^p \xi_n(\dd x,\dd y) = a^{pn} \int_{\rr^d\times\rr^d} |x-y|^p \pi_n(\dd x,\dd y)
=a^{pn} \cT_p(\cR_{n}^a\mu,\cR_{n}^a\nu).
\end{align}
Next, we introduce $q=\frac12 \sum_{n\geq 0} |\mu(G_n^a)-\nu(G_n^a)|$
and we define
\begin{align}\label{jab1}
\xi(\dd x,\dd y)=\sum_{n\geq 0} (\mu(G_n^a)\land\nu(G_n^a))\xi_n(\dd x,\dd y) + q\alpha(\dd x)\beta(\dd y),
\end{align}
where
$$
\alpha(\dd x)= \frac1q\sum_{n\geq 0}(\mu(G_n^a)-\nu(G_n^a))_+ \frac{\mu\vert_{G_n^a}(\dd x)}{\mu(G_n^a)} \quad \hbox{and} \quad
\beta(\dd y)= \frac1q\sum_{n\geq 0}(\nu(G_n^a)-\mu(G_n^a))_+ \frac{\nu\vert_{G_n^a}(\dd y)}{\nu(G_n^a)}.
$$
Using that $(G_n^a)_{n\geq 0}$ is a partition of $\rd$, that
$\xi_n\in\cH(\mu\vert_{G_n^a}/\mu(G_n^a),\nu\vert_{G_n^a}/\nu(G_n^a))$ and that
$$
q=\sum_{n\geq 0}(\nu(G_n^a)-\mu(G_n^a))_+=\sum_{n\geq 0}(\mu(G_n^a)-\nu(G_n^a))_+
=1-\sum_{n\geq 0} (\nu(G_n^a)\land\mu(G_n^a)),
$$
it is easily checked that $\alpha$ and $\beta$ are probability measures and that $\xi \in \cH(\mu,\nu)$.
Furthermore, setting $c_p=1\lor 2^{p-1}$,
\begin{align*}
\int_{\rr^d\times\rr^d} \!|x-y|^p \alpha(\dd x)\beta(\dd y)
\!\leq& c_p\int_{\rr^d\times\rr^d} \!(|x|^p +|y|^p)
\alpha(\dd x)\beta(\dd y)\!=c_p\!\int_{\rr^d}\! |x|^p \alpha(\dd x) +  c_p \!\int_{\rr^d}\! |y|^p \beta(\dd y).
\end{align*}
We have $|x|< a^n$ for all $x \in G_n^a$, whence
\begin{align}\label{jab3}
q\int_{\rr^d\times\rr^d}|x-y|^p \alpha(\dd x)\beta(\dd y)\leq & c_p
\sum_{n\geq 0} a^{pn} [(\mu(G_n^a)-\nu(G_n^a))_+ +(\nu(G_n^a)-\mu(G_n^a))_+ ] \notag\\
=& 2^p \e_p  \sum_{n\geq 0} a^{pn} |\mu(G_n^a)-\nu(G_n^a)|.
\end{align}
Using that $\cT_p(\mu,\nu)\leq \int_{\rd\times\rd} |x-y|^p \xi(\dd x,\dd y)$ 
and \eqref{jab1}-\eqref{jab2}-\eqref{jab3} 
completes the proof.
\end{proof}

We can now give the

\begin{proof}[Proof of Proposition \ref{depart}]
Fix $\mu$ and $\nu$ in $\cP(\rd)$ and $a>1$.
For each $n\geq 0$, the probability measures $\cR_n^a\mu$ and $\cR_n^a\nu$, defined in Lemma \ref{l2}, 
are supported in $B(0,1)$,
and $\cR_n^a\mu(C)=\frac{\mu(a^nC \cap G_n^a)}{\mu(G_n^a)}$ for all $C \in \cB(\rr^d)$.
Hence we know from Lemma \ref{cc} that there exists some numbers $r_{a,n,\ell}(\mu,\nu)=u_\ell(\cR_n^a\mu,\cR_n^a\nu)$
satisfying $\sum_{\ell=0}^{k-1} r_{a,n,\ell}(\mu,\nu)\leq 1$ and 
$$
r_{a,n,\ell}(\mu,\nu) \leq \frac12 \sum_{F\in \cQ_\ell}
\Big(\frac{\mu(a^nF \cap G_n^a)}{\mu(G_n^a)} \land \frac{\nu(a^nF \cap G_n^a)}{\nu(G_n^a)} \Big)
\sum_{C\!\hbox{\tiny{ child of }}\! F} \Big|\frac{\mu(a^nC \cap G_n^a)}{\mu(a^nF\cap G_n^a)} 
- \frac{\nu(a^nC \cap G_n^a)}{\nu(a^nF\cap G_n^a)}\Big|
$$
and such that
$$
\cT_p(\cR_n^a\mu,\cR_n^a\nu) \leq  \delta_k^p+\sum_{\ell=0}^{k-1} \delta_\ell^p r_{a,n,\ell}(\mu,\nu).
$$
Inserting this into \eqref{ccc} completes the proof.
\end{proof}

\section{A general estimate concerning the empirical measure}\label{gume}

To go further, we need a more precise setting.

\begin{set}\label{ssm}
Same points (a) and (b) as in Setting \ref{sss}.
\vip
(c) There are some constants $A,D>0$ and $r>1$ such that for each $k\geq 1$,
there is a family $(\cQ_{k,\ell})_{\ell=0,\dots,k}$ of nested partitions of $G_0$ such that $\cQ_{k,0}=\{G_0\}$
and such that 
\begin{gather}\label{card}
\forall\; \ell=1,\dots,k,\qquad |\cQ_{k,\ell}|\leq A \,r^{d\ell}
\end{gather}
and
\begin{gather}\label{card2}
\forall\; \ell=0,\dots,k,\qquad \delta_{k,\ell} = \max_{C \in \cQ_{k,\ell}} \sup_{x,y \in C} |x-y| \leq D r^{-\ell}.
\end{gather}
\end{set}
Recall that $\e_p=2^{-1}\lor 2^{-p}$.
The goal of this section is to prove the following result.

\begin{prop}\label{ccco}
We adopt Setting \ref{ssm}, consider $\mu \in \cP(\rd)$ and
the associated empirical measure $\mu_N$, see \eqref{muN}. 
Fix $p>0$ and assume that $\cM_q(\mu)<\infty$ for some $q>p$.
For all $a>1$, all $N\geq 1$,
$$
\E[\cT_p(\mu_N,\mu)]\leq K_{N}+\min\{L_{N},M_{N}\},
$$
where
\begin{align*}
K_{N}=&2^p\e_p \Big[\Big(2[1-\mu(G_0^a)]\Big)\land\sqrt{\frac{1-\mu(G_0^a)}{N}}\Big]
+ 2^p\e_p\sum_{n\geq 1} a^{pn} \Big[\Big(2\mu(G_n^a)\Big)\land \sqrt{\frac{\mu(G_n^a)}{N}}\Big],\\
L_{N}=& \frac{D^p\sqrt{Ar^d}}{2\sqrt N} \sum_{\ell\geq 0} r^{(d/2-p)\ell}
\sum_{n\geq 0} a^{pn}\sqrt{\mu(G_n^a)},\\
M_{N}=&D^p(1-r^{-p}) \sum_{n\geq 0} a^{pn}\sum_{\ell\geq 0} r^{-p\ell} \Big(\mu(G_n^a) 
\land \Big[\frac{\sqrt A r^{d\ell/2+d}}{2(r^{d/2} -1)}
\sqrt{\frac{\mu(G_n^a)}{N}}\Big] \Big).
\end{align*}
\end{prop}

We simply write $K_N,L_N,M_N$ for readability, but these quantities also depend on $p$, $a$ and $\mu$.

\begin{proof} 
We fix $\mu\in\cP(\rr^d)$, $a>1$ and $p>0$.
We also fix $k\geq 1$; we will let $k\to \infty$ at the end of the proof.
Applying Proposition \ref{depart}, with the family $(\cQ_{k,\ell})_{\ell=0,\dots,k}$, 
with $\mu$ and $\nu=\mu_N$ and taking expectations, we find
\begin{align}\label{rr1}
\E[\cT_p(\mu_N,\mu)]\leq U_{N,k}+V_{N,k}+W_{N,k}, 
\end{align}
where, setting 
$\rho_{k,a,n,\ell}=\E[(\mu(G_n^a)\land\mu_N(G_n^a))r_{k,a,n,\ell}(\mu,\mu_N)]$
(with $r_{k,a,n,\ell}(\mu,\mu_N)$ as defined in Proposition \ref{depart} with 
the family  $(\cQ_{k,\ell})_{\ell=0,\dots,k}$),
\begin{align}
U_{N,k}=&2^p\e_p \sum_{n\geq 0} a^{pn}\E[|\mu(G_n^a)-\mu_N(G_n^a)|],\notag\\
V_{N,k}=&\sum_{n\geq 0}a^{pn}\sum_{\ell= 0}^{k-1} \delta_{k,\ell}^p\rho_{k,a,n,\ell}\leq D^p
\sum_{n\geq 0}a^{pn}\sum_{\ell=0}^{k-1} r^{-p\ell}\rho_{k,a,n,\ell},\notag\\
W_{N,k}=& \delta_{k,k}^p\sum_{n\geq 0} a^{pn} \mu(G_n^a)\leq D^p r^{-pk} \sum_{n\geq 0} a^{pn}\mu(G_n^a). \label{rr2}
\end{align}
We used that $\delta_{k,\ell}\leq D r^{-\ell}$ for all $\ell\in\{0,\dots,k\}$, see \eqref{card2}.

\vip

Since $N\mu_N(G_n^a)$ is Binomial$(N,\mu(G_n^a))$-distributed, it holds that $\E[\mu_N(G_n^a)]=\mu(G_n^a)$ and 
$\Var [\mu_N(G_n^a)]=N^{-1}\mu(G_n^a)(1-\mu(G_n^a))$, from which we  deduce that
$$
\E[|\mu(G_n^a)-\mu_N(G_n^a)|]\leq \Big(2[1-\mu(G_n^a)]\Big)\land \Big(2\mu(G_n^a)\Big)
\land \sqrt{\frac{\mu(G_n^a)(1-\mu(G_n^a))}{N}}.
$$
We used that $|x-y|=|(1-x)-(1-y)|\leq 1-x+1-y$ for all $x,y \in [0,1]$ for the first bound,
that $|x-y|\leq x+y$ for the second one, and the Bienaym\'e-Tchebychev inequality for the third one.
All this implies that for all $k\geq 1$,
\begin{align}\label{rr3}
U_{N,k}\leq K_{N}.
\end{align}

Next, we observe that $\sum_{\ell=0}^{k-1}\rho_{k,a,n,\ell} \leq \mu(G_n^a)$ by \eqref{rnl1} and we claim that
for $\ell=0,\dots,k-1$,
\begin{align}\label{tp}
\rho_{k,a,n,\ell} \leq \frac12\sqrt{\frac{|\cQ_{k,\ell+1}|\mu(G_n^a)}{N}}.
\end{align}
Recalling \eqref{rnl2} and using that $(\mu(G_n^a)\land\mu_N(G_n^a))(\frac{\mu(a^nF\cap G_n^a)}{\mu(G_n^a)}\land
\frac{\mu_N(a^nF\cap G_n^a)}{\mu_N(G_n^a)}) \leq \mu_N(a^nF\cap G_n^a)$,
\begin{align}\label{tbi}
\rho_{k,a,n,\ell}\leq \frac12\sum_{F\in \cQ_{k,\ell}} \hskip3pt\sum_{C \!\hbox{\tiny{ child of }}\! F}
\E\Big[\Big|\mu_N(a^nC \cap G_n^a)\!-\frac{\mu_N(a^nF \cap G_n^a)\mu(a^nC \cap G_n^a)}{\mu(a^nF \cap G_n^a)}\Big|\Big].
\end{align}
But for $C$ a child of $F$, the 
conditional law
of $N\mu_N(a^nC \cap G_n^a)$ knowing that $N\mu_N(a^nF \cap G_n^a)=i$ is Binomial$(i,\frac{\mu(a^nC \cap G_n^a)}
{\mu(a^nF \cap G_n^a)})$, whence
$$
\E\Big[\Big|\mu_N(a^nC \cap G_n^a)-\frac{\mu_N(a^nF \cap G_n^a)\mu(a^nC \cap G_n^a)}{\mu(a^nF \cap G_n^a)}\Big|
\hskip3pt\Big| N\mu_N(a^nF \cap G_n^a)=i\Big] \leq \sqrt{\frac i {N^2} 
\frac{\mu(a^nC \cap G_n^a)}{\mu(a^nF \cap G_n^a)}.}
$$
Hence, since $\E[\sqrt{N\mu_N(A)}]\leq \sqrt{N\mu(A)}$ because $\E[\mu_N(A)]=\E[\mu(A)]$,
\begin{align*}
\E\Big[\Big|\mu_N(a^nC \cap G_n^a)\!-\frac{\mu_N(a^nF \cap G_n^a)\mu(a^nC \cap G_n^a)}{\mu(a^nF \cap G_n^a)}\Big|
\Big] \!\leq& \sqrt{\frac{\mu(a^nC \cap G_n^a)}{N^2\mu(a^nF \cap G_n^a)}} \E\Big[ \sqrt{N\mu_N(a^nF \cap G_n^a)}\Big]\\
\leq &\sqrt{\frac{\mu(a^nC \cap G_n^a)}{N}}.
\end{align*}
This, inserted in \eqref{tbi}, proves the claim \eqref{tp}, since by the Cauchy-Schwarz inequality,
$$
\sum_{F\in \cQ_{k,\ell}} \hskip3pt\sum_{C \!\hbox{\tiny{ child of }}\! F} \sqrt{\frac{\mu(a^nC \cap G_n^a)}{N}}
=\sum_{C \in \cQ_{k,\ell+1}} \sqrt{\frac{\mu(a^nC \cap G_n^a)}{N}} \leq \sqrt{\frac{|\cQ_{k,\ell+1}|\mu(G_n^a)}{N}}.
$$

We deduce from \eqref{tp} and \eqref{card} that
$$
\sum_{\ell=0}^{k-1} r^{-p\ell} \rho_{k,a,n,\ell}
\leq \frac12\sum_{\ell=0}^{k-1} r^{-p\ell}\sqrt{\frac{|\cQ_{k,\ell+1}|\mu(G_n^a)}{N}}
\leq\! \frac{\sqrt {A\mu(G_n^a)}}{2\sqrt{N}} \sum_{\ell=0}^{k-1}r^{-p\ell}r^{d(\ell+1)/2}.
$$
Since $V_{N,k}\leq D^p\sum_{n\geq 0}a^{pn}\sum_{\ell=0}^{k-1}r^{-p\ell} \rho_{k,a,n,\ell}$, we conclude that 
for all $k\geq 1$,
\begin{align}\label{rr4}
V_{N,k} \leq 
D^p\sum_{n\geq 0}a^{pn} \frac{\sqrt {A\mu(G_n^a)}}{2\sqrt{N}} \sum_{\ell\geq0}r^{-p\ell}r^{d(\ell+1)/2}
=L_{N}
\end{align}

Next, we set $S_{k,a,n,\ell}=\sum_{i=0}^\ell \rho_{k,a,n,i}$ for $\ell=0,\dots,k-1$ and $S_{k,a,n,-1}=0$ to write
$$
\sum_{\ell=0}^{k-1} r^{-p\ell}\rho_{k,a,n,\ell}=
\sum_{\ell=0}^{k-1} r^{-p\ell}(S_{k,a,n,\ell}-S_{k,a,n,\ell-1})
=(1-r^{-p})\sum_{\ell=0}^{k-1}r^{-p\ell}S_{k,a,n,\ell}+r^{-pk}S_{k,a,n,k-1}.
$$
But for each $\ell=0,\dots,k-1$, we both have $S_{k,a,n,\ell} \leq \mu(G_n^a)$ 
(since $\sum_{\ell=0}^{k-1}\rho_{k,a,n,\ell} \leq \mu(G_n^a)$, as already seen) and, by \eqref{tp} and \eqref{card},
$$
S_{k,a,n,\ell} \leq \frac 12 \sum_{i=0}^\ell \sqrt{\frac{|\cQ_{k,i+1}|\mu(G_n^a)}{N}}
\leq \frac{\sqrt{A \mu(G_n^a)}}{2\sqrt{N}} \sum_{i=0}^\ell r^{d(i+1)/2}\leq
\frac{\sqrt A r^{d\ell/2+d}}{2(r^{d/2} -1)}
\sqrt{\frac{\mu(G_n^a)}{N}}.
$$
Hence  
$$
\sum_{\ell=0}^{k-1} r^{-p\ell}\rho_{k,a,n,\ell}
\leq (1-r^{-p}) \sum_{\ell=0}^{k-1} r^{-p\ell} \Big(\mu(G_n^a) 
\land \Big[\frac{\sqrt A r^{d\ell/2+d}}{2(r^{d/2} -1)}
\sqrt{\frac{\mu(G_n^a)}{N}}\Big] \Big)+r^{-pk}\mu(G_n^a).
$$
Recalling that  $V_{N,k}\leq D^p\sum_{n\geq 0}a^{pn}\sum_{\ell=0}^{k-1}r^{-p\ell} \rho_{k,a,n,\ell}$, we
conclude that for all $k\geq 1$, it holds that
\begin{align}
V_{N,k} \leq & D^p\sum_{n\geq 0}a^{pn}(1-r^{-p}) \sum_{\ell \geq 0} r^{-p\ell} \Big(\mu(G_n^a) 
\land \Big[\frac{\sqrt A r^{d\ell/2+d}}{2(r^{d/2} -1)}
\sqrt{\frac{\mu(G_n^a)}{N}}\Big] \Big) + D^pr^{-pk}\sum_{n\geq 0} a^{pn}\mu(G_n^a)   \notag\\
=&M_{N} + D^pr^{-pk}\sum_{n\geq 0} a^{pn}\mu(G_n^a).\label{rr5}
\end{align}

Gathering \eqref{rr1}-\eqref{rr2}-\eqref{rr3}-\eqref{rr4}-\eqref{rr5}, we have proved that for all $k\geq 1$,
\begin{equation}\label{tti}
\E[\cT_p(\mu_N,\mu)]\leq K_N+\min\Big\{L_N,M_N+ D^pr^{-pk}\sum_{n\geq 0} a^{pn}\mu(G_n^a)\Big\}
+ D^pr^{-pk}\sum_{n\geq 0} a^{pn}\mu(G_n^a).
\end{equation}
Since $\mu(G_n^a)\leq \cM_q(\mu) a^{(1-n)q}$ for all $n\geq 1$
because $G_n^a \subset B(0,a^{n-1})^c$, and since $q>p$, we deduce that 
$\sum_{n\geq 0} a^{pn}\mu(G_n^a)<\infty$.
Letting $k\to \infty$ in \eqref{tti} thus completes the proof.
\end{proof}

Let us mention that the penultimate paragraph of this proof, where we handle a discrete integration
by parts, is crucial to obtain reasonable constants when $p\in(0,d/2)$.

\section{Precise study of some series}\label{series}

\begin{lem}\label{phi}
Fix $r>1$, $\beta\geq \alpha>0$ and $x \geq 0$ and put
$$
\Psi_{r,\alpha,\beta}(x)= \sum_{\ell \geq 0} r^{-\alpha \ell} [1 \land (x \;r^{\beta \ell})].
$$ 
With the notation $\log_+ x= (\log x)\lor 0$, it holds that
\begin{align}
\Psi_{r,\alpha,\beta}(x) \leq& \Big(\frac{\log_+ (1/x)}{\beta\log r} + \frac{1}{1-r^{-\alpha}}\Big)x
&\hbox{if $\beta=\alpha$,}\label{tt2}\\
\Psi_{r,\alpha,\beta}(x) \leq & 
\Big(\frac 1 {r^{\beta-\alpha}-1} + \frac 1{1-r^{-\alpha}} \Big) x^{\alpha/\beta}
 &\hbox{if $\beta>\alpha$.}\label{tt3}
\end{align}
\end{lem}

\begin{proof}
We fix $\alpha=\beta>0$ and prove \eqref{tt2}. 
If $x>1$, we write 
$$
\Psi_{r,\alpha,\beta}(x)\leq \sum_{\ell \geq 0} r^{-\alpha \ell}= \frac 1{1-r^{-\alpha}}\leq \frac x{1-r^{-\alpha}}.
$$
If $x\in [0,1]$, we set $t_x= \log (1/x) / (\beta \log r) \geq 0$, $\ell_x=\lfloor t_x \rfloor\in\nn$ and 
$s_x=t_x-\ell_x\in [0,1)$ and write
$$
\Psi_{r,\alpha,\beta}(x)\leq \sum_{\ell=0}^{\ell_x} x + \sum_{\ell=\ell_x+1}^{\infty} r^{-\alpha\ell}
=x(\ell_x+1)+\frac{r^{-\alpha(\ell_x+1)}}{1-r^{-\alpha}}
= u+v,
$$
where $u=x t_x +x/(1-r^{-\alpha})$ is the  desired bound and where, since $x=r^{-\alpha t_x}$,
$$
v=x(\ell_x+1-t_x)+ \frac{r^{-\alpha(\ell_x+1)}-x}{1-r^{-\alpha}}=x
\Big[1-s_x + \frac{r^{-\alpha(1-s_x)}-1}{1-r^{-\alpha}}\Big].
$$
To show that $v\leq 0$, which will complete the proof of \eqref{tt2}, 
it suffices to prove that $g(u)=u+ \frac{r^{-\alpha u}-1}{1-r^{-\alpha}}$
is nonpositive for all $u\in[0,1]$. But
$g''(u)=\frac{(\alpha \log r)^2}{1-r^{-\alpha}}r^{-\alpha u}\geq 0$, so that
$g$ is convex, and $g(0)=g(1)=0$. The conclusion follows.

\vip

We fix $\beta>\alpha>0$ and prove \eqref{tt3}.
If $x>1$, we write 
$$
\Psi_{r,\alpha,\beta}(x)\leq \sum_{\ell \geq 0} r^{-\alpha \ell}= \frac 1{1-r^{-\alpha}}\leq \frac {x^{\alpha/\beta}}{1-r^{-\alpha}}.
$$
If $x\in [0,1]$, we set $t_x= \log (1/x) / (\beta \log r)\geq 0$, $\ell_x=\lfloor t_x \rfloor\in\nn$ and 
$s_x=t_x-\ell_x\in [0,1)$ and write
$$
\Psi_{r,\alpha,\beta}(x)\leq \sum_{\ell=0}^{\ell_x} x r^{(\beta-\alpha)\ell} + \sum_{\ell=\ell_x+1}^{\infty} r^{-\alpha\ell}
\leq x \frac{r^{(\beta-\alpha)(\ell_x+1)}}{r^{\beta-\alpha}-1}+\frac{r^{-\alpha(\ell_x+1)}}{1-r^{-\alpha}}
= u+v,
$$
where $u=x^{\alpha/\beta}/(r^{\beta-\alpha}-1)+x^{\alpha/\beta}/(1-r^{-\alpha})$ is the desired bound and where,
since $x=r^{-\beta t_x}$,
$$
v=\frac{x r^{(\beta-\alpha)(\ell_x+1)}-x^{\alpha/\beta}}{r^{\beta-\alpha}-1}+\frac{r^{-\alpha(\ell_x+1)}-x^{\alpha/\beta}}{1-r^{-\alpha}}
=x^{\alpha/\beta}\Big[\frac{r^{(\beta-\alpha)(1-s_x)}-1}{r^{\beta-\alpha}-1}+\frac{r^{-\alpha(1-s_x)}-1}{1-r^{-\alpha}} \Big].
$$
To show that $v\leq 0$, which will complete the proof of \eqref{tt3}, 
it suffices to show that $g(u)=\frac{r^{(\beta-\alpha)u}-1}{r^{\beta-\alpha}-1}+\frac{r^{-\alpha u}-1}{1-r^{-\alpha}}$
is nonpositive for all $u\in[0,1]$. But $g''(u)=\frac{((\beta-\alpha) \log r)^2}{r^{\beta-\alpha}-1}r^{(\beta-\alpha)u}
+\frac{(\alpha \log r)^2}{1-r^{-\alpha}}r^{-\alpha u} \geq 0$, so that
$g$ is convex, and it holds that $g(0)=g(1)=0$.
\end{proof}

\section{Theoretical result for a general norm}\label{titrenaze}

Recall that $\e_p=2^{-1}\lor 2^{-p}$ and that $H$ was defined in \eqref{H}. 
Here we prove the following general result,
to be applied to some specific norms later.

\begin{prop}\label{genen}
We adopt Setting \ref{ssm}, we fix $\mu \in \cP(\rd)$ and consider the associated empirical 
measure $\mu_N$, see \eqref{muN}. We fix $q>p>0$ and assume that $\cM_q(\mu)<\infty$.
\vip
(i) If $p>d/2$ and $q>2p$, then for all $N\geq 1$,
$$
\E[\cT_p(\mu_N,\mu)]\leq 2^p\frac{\kappa_{d,p}}{\sqrt N}[\cM_q(\mu)]^{p/q}H\Big(\frac{\e_p}{\kappa_{d,p}},2p,q\Big),
$$
where
$$
\kappa_{d,p}=\frac{D^p\sqrt{A r^d}}{2^{p+1}(1-r^{d/2-p})}.
$$

(ii) If $p=d/2$ and $q>2p$, then for all $N\geq 1$,
$$
\E[\cT_p(\mu_N,\mu)]\leq 2^p\frac{\kappa_{d,p,N}}{\sqrt N}[\cM_q(\mu)]^{p/q}H\Big(\frac{\e_p}{\kappa_{d,p,N}},2p,q\Big),
$$
where
$$
\kappa_{d,p,N}=\frac{D^p\sqrt{A} r^{p}}{2^{p+1}p\log r}\log_+\Big(2(r^{-p}-r^{-2p})\sqrt{\frac N A}\Big) 
+\frac{D^p\sqrt{A} r^{2p}}{2^{p+1}(r^p-1)}.
$$

(iii) If $p\in (0,d/2)$ and $q>dp/(d-p)$, then for all $N\geq 1$,
$$
\E[\cT_p(\mu_N,\mu)]\leq 2^p\frac{\kappa_{d,p}}{N^{p/d}}[\cM_q(\mu)]^{p/q}
H\Big(\frac{2^{1-2p/d}\e_p}{\kappa_{d,p}},\frac{dp}{d-p},q\Big),
$$
where
$$
\kappa_{d,p}=\frac{D^pA^{p/d}r^{p}(r^{d/2}-1)^{1-2p/d}}{2^{p+2p/d}(r^{d/2-p}-1)}.
$$
\end{prop}

\begin{proof}
We fix $q>p>0$. We have $(G_0^a)^c \subset \{x \in \rd : |x|\geq 1\}$ and
$G_n^a \subset  \{x \in \rd : |x|\geq a^{n-1}\}$ for each $n\geq 1$, whence
\begin{equation}\label{mom}
1-\mu(G_0^a) \leq \cM_q(\mu) \qquad \hbox{and}\qquad \mu(G_n^a) \leq \cM_q(\mu)a^{-q(n-1)}\quad\hbox{if $n\geq 1$}.
\end{equation}
We know from Proposition \ref{ccco} that 
$\E[\cT_p(\mu_N,\mu)] \leq K_{N}+\min\{L_{N},M_{N}\}$.
\vip
{\it Case (i): $p>d/2$ and $q>2p$.} First, by \eqref{mom}, we have
$$
K_{N}\leq 2^p\e_p\sqrt{\frac{\cM_q(\mu)}{N}}+ 2^p\e_p\sum_{n\geq 1} a^{pn}\sqrt{\frac{\cM_q(\mu)}{Na^{q(n-1)}}}
= 2^p\e_p \sqrt{\frac{\cM_q(\mu)}N}\Big[1+\frac {a^p}{1-a^{p-q/2}}\Big].
$$
Next,
\begin{align*}
L_{N}\leq& \frac{D^p\sqrt{Ar^d}}{2\sqrt N} \sum_{\ell\geq 0}r^{(d/2-p)\ell}\Big( 1 + \sum_{n\geq 1}a^{pn}
\sqrt{\frac{\cM_q(\mu)}{a^{q(n-1)}}}\Big)=2^p\frac{\kappa_{d,p}}{\sqrt N}
\Big(1+\sqrt{\cM_q(\mu)} \frac {a^p}{1-a^{p-q/2}}\Big)
\end{align*}
recall that $\kappa_{d,p}=(D/2)^p\sqrt{Ar^d}/[2(1-r^{d/2-p})]$.
All in all, we have proved that
\begin{equation}\label{fri}
\E[\cT_p(\mu_N,\mu)]\leq \frac{2^p}{\sqrt N} \Big(\kappa_{d,p}+\sqrt{\cM_q(\mu)}\Big[\e_p+(\e_p+\kappa_{d,p})
 \frac {a^p}{1-a^{p-q/2}}\Big] \Big).
\end{equation}
This holds true for any value of $a>1$ and we optimally choose
$a=[q/(2p)]^{2/(q-2p)}$ and set
$$
v_{p,q}=\frac{a^p}{1-a^{p-q/2}}=
\frac{q}{q-2p}\Big(\frac{q}{2p}\Big)^{2p/(q-2p)}.
$$
We thus have
$$
\E[\cT_p(\mu_N,\mu)]\leq \frac{2^p}{\sqrt N} \Big(\kappa_{d,p}+\sqrt{\cM_q(\mu)}
[\e_p+(\e_p+\kappa_{d,p})v_{p,q}] \Big)
= \frac{2^p \kappa_{d,p}}{\sqrt N} \Big(1+\sqrt{\cM_q(\mu)}\rho_{d,p,q}\Big),
$$
where $\rho_{d,p,q}=\e_p/\kappa_{d,p}+(\e_p+\kappa_{d,p})v_{p,q}/\kappa_{d,p}$. 
\vip
For any $\alpha>0$, we may apply this formula to $\mu^\alpha$, the image measure of $\mu$
by the map $x \mapsto \alpha x$, which satisfies $\E[\cT_p(\mu_N^\alpha,\mu^\alpha)]=\alpha^p\E[\cT_p(\mu_N,\mu)]$
and $\cM_q(\mu^\alpha)=\alpha^q \cM_q(\mu)$. We thus get
$$
\E[\cT_p(\mu_N,\mu)]\leq \frac{2^p \kappa_{d,p}}{\sqrt N} \frac{1}{\alpha^p}
\Big(1+\sqrt{\alpha^q\cM_q(\mu)}\rho_{d,p,q}\Big).
$$
We optimally choose $\alpha=[(q-2p)\rho_{d,p,q}\sqrt{\cM_q(\mu)}/(2p)]^{-2/q}$ and find
\begin{align*}
\E[\cT_p(\mu_N,\mu)]\leq& \frac{2^p \kappa_{d,p}}{\sqrt N}
[\cM_q(\mu)]^{p/q} \Big(\rho_{d,p,q} \frac{q-2p}{2p}\Big)^{2p/q} \frac{q}{q-2p}\\
=&\frac{2^p \kappa_{d,p}}{\sqrt N}[\cM_q(\mu)]^{p/q} \Big(\frac{\e_p}{\kappa_{d,p}}\frac{q-2p}{2p}
+\frac{\e_p+\kappa_{d,p}}{\kappa_{d,p}}v_{p,q}\frac{q-2p}{2p}\Big)^{2p/q} \frac{q}{q-2p}\\
=&\frac{2^p \kappa_{d,p}}{\sqrt N}[\cM_q(\mu)]^{p/q} \Big(\frac{\e_p}{\kappa_{d,p}}\frac{q-2p}{2p}
+\frac{\e_p+\kappa_{d,p}}{\kappa_{d,p}}\Big(\frac{q}{2p}\Big)^{q/(q-2p)}\Big)^{2p/q} \frac{q}{q-2p}\\
=&\frac{2^p \kappa_{d,p}}{\sqrt N}[\cM_q(\mu)]^{p/q} H\Big(\frac{\e_p}{\kappa_{d,p}},2p,q\Big).
\end{align*}

{\it Case (ii): $p=d/2$ and $q>2p$.} Exactly as in Case (i),
$$
K_{N}\leq 2^p\e_p \sqrt{\frac{\cM_q(\mu)}N}\Big[1+\frac {a^p}{1-a^{p-q/2}}\Big].
$$
We next write
\begin{align*}
M_{N}\leq&D^p(1-r^{-p})\sum_{n\geq 0}a^{pn} \mu(G^a_n)
\sum_{\ell\geq 0} r^{-p\ell} \Big(1\land\Big[\frac{\sqrt A r^{d\ell/2+d}}{2(r^{d/2}-1)}
\frac1{\sqrt{N\mu(G_n^a)}}\Big]\Big)\\
= &D^p(1-r^{-p})\sum_{n\geq 0}a^{pn} \mu(G^a_n)
\Psi_{r,p,d/2}\Big(\frac{\sqrt A r^{d}}{2(r^{d/2}-1)\sqrt{N\mu(G_n^a)}}\Big).
\end{align*}
By \eqref{tt2} with $\alpha=\beta=p=d/2$, we can
bound $M_N$ by
\begin{align*}
&D^p(1-r^{-p})\sum_{n\geq 0}a^{pn} \mu(G^a_n)\Big[
\frac{\log_+(2(r^{-d/2}-r^{-d})\sqrt{N\mu(G_n^a)/A})}{p\log r}+\frac1{1-r^{-p}}\Big]
\frac{\sqrt A r^{d}}{2(r^{d/2}-1)\sqrt{N\mu(G_n^a)}}\\
\leq & \frac{D^p\sqrt A r^{2p}}{2(r^{p}-1)\sqrt N}\Big[\frac{(1-r^{-p})\log_+(2(r^{-p}-r^{-2p})\sqrt{N/A})}{p\log r}
+1\Big]\sum_{n\geq 0}a^{pn} \sqrt{\mu(G^a_n)}
\end{align*}
since $\mu(G_n^a)\leq 1$. Observing that, by \eqref{mom},
$$
\sum_{n\geq 0}a^{pn} \sqrt{\mu(G^a_n)}\leq 1+\sum_{n\geq 1}a^{pn}\sqrt{\frac{\cM_q(\mu)}{a^{q(n-1)}}}
= 1+\sqrt{\cM_q(\mu)}
\frac{a^p}{1-a^{p-q/2}}, 
$$
and recalling that
$$
\kappa_{d,p,N}=\frac{(D/2)^p\sqrt A r^{p}\log_+(2(r^{-p}-r^{-2p})\sqrt{N/A})}{2p\log r}
+\frac{(D/2)^p\sqrt A r^{2p}}{2(r^p-1)},
$$
we conclude that
$$
M_N \leq 2^p \frac{\kappa_{d,p,N}}{\sqrt N}\Big(1+\sqrt{\cM_q(\mu)}\frac{a^p}{1-a^{p-q/2}}\Big).
$$
All in all, we have proved that
$$
\E[\cT_p(\mu_N,\mu)]\leq \frac{2^p}{\sqrt N}\Big(\kappa_{d,p,N}+\sqrt{\cM_q(\mu)}\Big[\e_p+(\e_p+\kappa_{d,p,N})
\frac{a^p}{1-a^{p-q/2}}\Big] \Big).
$$
From there we conclude exactly as in Case (i) (compare the above formula to \eqref{fri}) that
$$
\E[\cT_p(\mu_N,\mu)]\leq 2^p\frac{\kappa_{d,p,N}}{\sqrt N}[\cM_q(\mu)]^{p/q}H\Big(\frac{\e_p}{\kappa_{d,p,N}},2p,q\Big).
$$

{\it Case (iii): $p\in (0,d/2)$ and $q>dp/(d-p)$.} We write, using that $2p/d\in (0,1)$ and then \eqref{mom},
\begin{align*}
K_N \leq &2^p\e_p \Big[2(1-\mu(G_0^a))\Big]^{1-2p/d}\Big[\sqrt{\frac{1-\mu(G_0^a)}{N}}\Big]^{2p/d}+
2^p\e_p\sum_{n\geq 1}a^{pn} \Big[2\mu(G_n^a)\Big]^{1-2p/d}
\Big[\sqrt{\frac{\mu(G_n^a)}{N}}\Big]^{2p/d}\\
\leq& 2^p\e_p    \Big[2\cM_q(\mu)\Big]^{1-2p/d}\Big[\sqrt{\frac{\cM_q(\mu)}{N}}\Big]^{2p/d}
+2^p\e_p\sum_{n\geq 1}a^{pn} \Big[\frac{2\cM_q(\mu)}{a^{q(n-1)}}\Big]^{1-2p/d}
\Big[\sqrt{\frac{\cM_q(\mu)}{Na^{q(n-1)}}}\Big]^{2p/d}\\
= &2^p\e_p \frac{2^{1-2p/d}[\cM_q(\mu)]^{1-p/d}}{N^{p/d}}
+2^p\e_p \frac{2^{1-2p/d}[\cM_q(\mu)]^{1-p/d}}{N^{p/d}}
\sum_{n\geq 1}a^{pn-q(1-p/d)(n-1)}\\
= & 2^p\e_p \frac{2^{1-2p/d}[\cM_q(\mu)]^{1-p/d}}{N^{p/d}} \Big(1+\frac{a^p}{1-a^{p-q+pq/d}}\Big).
\end{align*}
We used that $p-q+pq/d<0$ because $q>dp/(d-p)$. Next,
\begin{align*}
M_{N}\leq D^p(1-r^{-p})\sum_{n\geq 0}a^{pn} \mu(G^a_n)
\Psi_{r,p,d/2}\Big(\frac{\sqrt A r^{d}}{2(r^{d/2}-1)\sqrt{N\mu(G_n^a)}}\Big)
\end{align*}
as in Case (ii). Thus, using \eqref{tt3} with $\alpha=p$ and $\beta=d/2$,
\begin{align*}
M_N\leq & D^p(1-r^{-p})\Big[\frac1{r^{d/2-p}-1}+\frac1{1-r^{-p}}\Big]
\sum_{n\geq 0}a^{pn} \mu(G^a_n) \Big(\frac{\sqrt A r^{d}}{2(r^{d/2}-1)\sqrt{N\mu(G_n^a)}}\Big)^{2p/d}\\
=& 2^p \frac{\kappa_{d,p}}{N^{p/d}}\sum_{n\geq 0}a^{pn} [\mu(G^a_n)]^{1-p/d},
\end{align*}
because
\begin{align*}
\Big(\frac D2\Big)^p(1-r^{-p})\Big[\frac1{r^{d/2-p}-1}+\frac1{1-r^{-p}}\Big]
\Big(\frac{\sqrt A r^{d}}{2(r^{d/2}-1)}\Big)^{2p/d}=\frac{D^pA^{p/d}r^p(r^{d/2}-1)^{1-2p/d}}
{2^{p+2p/d}(r^{d/2-p}-1)}
=\kappa_{d,p}.
\end{align*}
But, using \eqref{mom},
$$
\sum_{n\geq 0}a^{pn} [\mu(G^a_n)]^{1-p/d} 
\leq  1+\sum_{n\geq 1} a^{pn}\Big[\frac{\cM_q(\mu)}{a^{q(n-1)}}\Big]^{1-p/d}
= 1+ [\cM_q(\mu)]^{1-p/d}\frac{a^p}{1-a^{p-q+pq/d}},
$$
whence
$$
M_N\leq 2^p \frac{\kappa_{d,p}}{N^{p/d}}\Big(1+ [\cM_q(\mu)]^{1-p/d}\frac{a^p}{1-a^{p-q+pq/d}}\Big).
$$
All in all, we have
\begin{align*}
\E[\cT_p(\mu_N,\mu)]\leq&
\frac{2^p}{N^{p/d}}\Big(\kappa_{d,p}+[\cM_q(\mu)]^{1-p/d}\Big[2^{1-2p/d}\e_p+
(2^{1-2p/d}\e_p+\kappa_{d,p})\frac{a^p}{1-a^{p-q+pq/d}}\Big]\Big)\\
=&\frac{2^p}{N^{p/d}}\Big(\kappa_{d,p} + [\cM_q(\mu)]^{p/\tau}\Big[2^{1-2p/d}\e_p+
(2^{1-2p/d}\e_p+\kappa_{d,p}) \frac{a^p}{1-a^{p-pq/\tau}} \Big]\Big),
\end{align*}
where we have set $\tau=dp/(d-p)$.
We choose $a=[q/\tau]^{\tau/(p(q-\tau))}>1$, for which
$$
v_{d,p,q}=\frac{a^p}{1-a^{p-pq/\tau}}=\frac{q}{q-\tau}\Big(\frac{q}\tau\Big)^{\tau/(q-\tau)}.
$$
Thus
\begin{align*}
\E[\cT_p(\mu_N,\mu)]\leq &
\frac{2^p \kappa_{d,p}}{N^{p/d}}\Big(1
+[\cM_q(\mu)]^{p/\tau}\Big[\frac{2^{1-2p/d}\e_p}{\kappa_{d,p}}
+ (2^{1-2p/d}\e_p+\kappa_{d,p})\frac{v_{d,p,q}}{\kappa_{d,p}}\Big]\Big)\\
=&\frac{2^p \kappa_{d,p}}{N^{p/d}}\Big(1
+[\cM_q(\mu)]^{p/\tau}\rho_{d,p,q}\Big),
\end{align*} 
where $\rho_{d,p,q}\!=\!2^{1-2p/d}\e_p/\kappa_{d,p}+(2^{1-2p/d}\e_p+\kappa_{d,p})v_{d,p,q}/\kappa_{d,p}$.
As in Case (i), we deduce that
$$
\E[\cT_p(\mu_N,\mu)]\leq \frac{2^p\kappa_{d,p}}{N^{p/d}}\frac1{\alpha^p}\Big(1
+ [\alpha^q \cM_q(\mu)]^{p/\tau}\rho_{d,p,q}\Big)
$$
for all $\alpha>0$. 
With $\alpha=[\rho_{d,p,q}\cM_q^{p/\tau}(q-\tau)/\tau]^{-\tau/(pq)} $, which is optimal,
we find
\begin{align*}
\E[\cT_p(\mu_N,\mu)]\leq &\frac{2^p\kappa_{d,p}}{N^{p/d}}[\cM_q(\mu)]^{p/q}
\Big(\rho_{d,p,q}\frac{q-\tau}{\tau}\Big)^{\tau/q}\frac{q}{q-\tau}\\
=&\frac{2^p\kappa_{d,p}}{N^{p/d}}[\cM_q(\mu)]^{p/q}
\Big(\frac{2^{1-2p/d}\e_p}{\kappa_{d,p}}\frac{q-\tau}{\tau}
+\frac{2^{1-2p/d}\e_p+\kappa_{d,p}}{\kappa_{d,p}}v_{d,p,q}\frac{q-\tau}{\tau}\Big)^{\tau/q}\frac{q}{q-\tau}\\
=&\frac{2^p\kappa_{d,p}}{N^{p/d}}[\cM_q(\mu)]^{p/q}
\Big(\frac{2^{1-2p/d}\e_p}{\kappa_{d,p}}\frac{q-\tau}{\tau}
+ \frac{2^{1-2p/d}\e_p
+\kappa_{d,p}}{\kappa_{d,p}}\Big(\frac q\tau\Big)^{q/(q-\tau)}\Big)^{\tau/q}\frac{q}{q-\tau}\\
=&\frac{2^p\kappa_{d,p}}{N^{p/d}}[\cM_q(\mu)]^{p/q}H\Big(\frac{2^{1-2p/d}\e_p}{\kappa_{d,p}},\tau,q\Big)
\end{align*}
as desired.
\end{proof}

\section{Conclusion for the maximum norm}\label{concmax}

Here we consider the maximum norm $|\cdot|_\infty$. We claim that Setting \ref{ssm}-(c) holds
true with $A=1$, $D=2$ and $r=2$.
Indeed, consider, for each $\ell\geq 0$, the natural partition $\cQ_\ell$ of $G_0=[-1,1]^d$ into 
$2^{d\ell}$ translations of $[-2^{-\ell},2^{-\ell}]^d$ (we actually have to remove some
of the common faces, but this is of course not an issue).
Then for any $k\geq 1$, $(\cQ_\ell)_{\ell=0,\dots,k}$ is a family of nested partitions of $G_0$,
we have $|\cQ_\ell|=2^{d\ell} =Ar^{d\ell}$ for all $\ell=1,\dots,k$ and
$\delta_\ell=\max_{C \in \cQ_\ell} \sup_{x,y\in C}|x-y|= 2\times 2^{-\ell}=Dr^{-\ell}$
for all $\ell =0,\dots, k$.

\vip

We thus may apply Proposition \ref{genen} with these values $A=1$, $D=2$ and $r=2$. This gives Theorem
\ref{mr} (with the norm  $|\cdot|_\infty$) with the announced formulas, which we now check.

\vip
If first $p>d/2$, we have
$$
\kappa_{d,p}^{(\infty)}=\frac{D^p\sqrt{A r^d}}{2^{p+1}(1-r^{d/2-p})}=\frac{2^{d/2-1}}{1-2^{d/2-p}}.
$$

If next $p=d/2$, we have
\begin{align*}
\kappa_{d,p,N}^{(\infty)}=&\frac{D^p\sqrt{A} r^p}{2^{p+1}p \log r}
\log_+\Big(2(r^{-p}-r^{-2p})\sqrt{\frac N A}\Big) 
+\frac{D^p\sqrt{A} r^{2p}}{2^{p+1}(r^p-1)}\\
=&\frac{2^{p-1}}{p\log 2} \log_+\Big((2^{1-p}-2^{1-2p})\sqrt N \Big)+ \frac{2^{p-1}}{1-2^{-p}}.
\end{align*}

If finally $p\in (0,d/2)$, we have
$$
\kappa_{d,p}^{(\infty)}=
\frac{D^pA^{p/d}r^p(r^{d/2}-1)^{1-2p/d}}{2^{p+2p/d}(r^{d/2-p}-1)}
=\frac{2^{p-2p/d}(1-2^{-d/2})^{1-2p/d}}{1-2^{p-d/2}}.
$$

\section{Conclusion for the \blu other norms\bla}\label{conceuc}

\blu
We now work with $\mm \in [1,\infty)$.
The following lemma follows from Le Gouic \cite[Lemma 3.18]{lg}
Recall that $K_d^{(\mm)}$ was defined in \eqref{Kd} and that 
$B_\mm(x,r)=\{y\in\rd : |y-x|_\mm<r\}$.

\begin{lem}
For any $k\geq 1$, any $r\geq 2$, there exists a family $(\cQ_{k,\ell})_{\ell=0,\dots,k}$
of nested partitions of $B_\mm(0,1)$ such that $\cQ_{k,0}=\{B_\mm(0,1)\}$,
with $|\cQ_{k,\ell}|\leq K_d^{(\mm)} 2^{-d}r^{d\ell}$ for all $\ell=1,\dots,k$ and
$\delta_{k,\ell}=\max_{C \in \cQ_{k,\ell}} \sup_{x,y \in C} |x-y|_\mm \leq (4r/(r-1)) r^{-\ell}$ for all $\ell=0,\dots,k$.
\end{lem}

It suffices to use \cite[Lemma 3.18]{lg} with $E=B_\mm(0,1)$, $d(x,y)=|x-y|_\mm$, 
$D=2$, $\e=r^{-1}$ and to note that for 
$\ell=1,\dots,k$, since $|\cQ_{k,\ell}|\leq N_{2r^{-\ell}}^{(\mm)}$ and $2r^{-\ell}\in (0,1]$, we have 
$|\cQ_{k,\ell}|\leq K_d^{(\mm)} 2^{-d}r^{d\ell}$.

\vip

Thus Setting \ref{ssm}-(c) holds with any $r\geq 2$, with $A=K_d^{(\mm)} 2^{-d}$ and $D=4r/(r-1)$,
so that we may apply Proposition \ref{genen} with these values. Optimizing in $r\geq 2$,
this gives Theorem
\ref{mr} (with the norm  $|\cdot|_\mm$) with the announced formulas, which we now check.

\vip
If first $p>d/2$, we find $\kappa_{d,p}^{(\mm)}=\min\{\kappa_{d,p,r}^{(\mm)} : r\geq 2\}$, where
$$
\kappa_{d,p,r}^{(\mm)}=\frac{D^p\sqrt{A r^d}}{2^{p+1}(1-r^{d/2-p})}
=\sqrt{K_d^{(\mm)}}\frac{2^{p-1-d/2}r^{p+d/2}}{(r-1)^p(1-r^{d/2-p})}.
$$

If next $p=d/2$, we have $\kappa_{d,p,N}^{(\mm)}=\min\{\kappa_{d,p,N,r}^{(\mm)} : r\geq 2\}$, where
\begin{align*}
\kappa_{d,p,N,r}^{(\mm)}=&\frac{D^p\sqrt{A} r^{p}}{2^{p+1}p\log r}
\log_+\Big(2(r^{-p}-r^{-2p})\sqrt{\frac N A}\Big) 
+\frac{D^p\sqrt{A} r^{2p}}{2^{p+1}(r^p-1)}\\
=& \frac{\sqrt{K_d^{(\mm)}}}2 \frac{r^{2p}}{(r-1)^p p\log r} 
\log_+\Big(2^{p+1}(r^{-p}-r^{-2p})\sqrt{\frac{N}{K_d^{(\mm)}}} \Big)
+ \frac{\sqrt{K_d^{(\mm)}}}2 \frac{r^{3p}}{(r-1)^p (r^p-1)}.
\end{align*}

If finally $p\in (0,d/2)$, we have $\kappa_{d,p}^{(\mm)}=\min\{\kappa_{d,p,r}^{(\mm)} : r\geq 2\}$, where
$$
\kappa_{d,p,r}^{(\mm)}=
\frac{D^pA^{p/d}r^p(r^{d/2}-1)^{1-2p/d}}{2^{p+2p/d}(r^{d/2-p}-1)}
= \Big(\frac{K_d^{(\mm)}}{4}\Big)^{p/d} \frac{r^{2p} (1-r^{-d/2})^{1-2p/d}}{(r-1)^p(1-r^{p-d/2})}.
$$

\bla

\section{The case of a low order finite moment}\label{lom}
We finally handle the case where $\mu$ has a low order moment.
We only treat the case of the maximum norm for simplicity.
We thus may apply Proposition \ref{ccco} with
$A=1$, $D=2$ and $r=2$, see the beginning of Section \ref{concmax}.
\vip

\begin{proof}[Proof of Theorem \ref{mrn}]
We consider $p>0$, $q\in (p,\min\{2p,dp/(d-p)\})$, $\mu\in \cP(\rd)$ and the associated empirical
measure $\mu_N$.
We know that $\E[\cT_p^{(\infty)}(\mu_N,\mu)] \leq K_N+\min\{L_{N},M_N\}$
by Proposition \ref{ccco}, and we have 
\begin{equation}\label{mom2}
\mu(G_n^a) \leq \cM_q^{(\infty)}(\mu)a^{-q(n-1)}\quad\hbox{if $n\geq 1$}.
\end{equation}
as usual. First,
\begin{align*}
K_N\leq & \frac{2^p\e_p}{\sqrt N} + 2^p\e_p\sum_{n\geq 1} a^{pn}\Big[\frac{2\cM_q^{(\infty)}(\mu)}{a^{q(n-1)}}
\land \sqrt{\frac{\cM_q^{(\infty)}(\mu)}{N a^{q(n-1)}}} \Big]\\
=& \frac{2^p\e_p}{\sqrt N}+ 2^p\e_pa^p  \sum_{n\geq 0} a^{(p-q)n}\Big[\Big(2\cM_q^{(\infty)}(\mu)\Big)
\land \Big(\sqrt{\frac{\cM_q^{(\infty)}(\mu)}{N}} a^{qn/2}\Big) \Big]\\
=& \frac{2^p\e_p}{\sqrt N}+2^{p+1}\e_p\cM_q^{(\infty)}(\mu) a^p   
\Psi_{a,q-p,q/2}\Big(\frac 1{2\sqrt{N\cM_q^{(\infty)}(\mu)}}\Big).
\end{align*}
Since $q/2>q-p$ because $q<2p$, we may apply \eqref{tt3} with $r=a$, with $\alpha=q-p$ and $\beta=q/2$:
\begin{align*}
K_N\leq & \frac{2^p\e_p}{\sqrt N}+ 2^{p+1}\e_p\cM_q^{(\infty)}(\mu) a^p\Big[\frac1{a^{p-q/2}-1}+\frac1{1-a^{p-q}}\Big]
\Big(\frac 1{2\sqrt{N\cM_q^{(\infty)}(\mu)}}\Big)^{2(q-p)/q}\\
=&\frac{2^p\e_p}{\sqrt N} + 2^p\e_p \rho_a [\cM_q^{(\infty)}(\mu)]^{p/q}\frac{2^{2p/q-1}}{N^{(q-p)/q}},
\end{align*}
where $\rho_a=a^p[1/(a^{p-q/2}-1)+1/(1-a^{p-q})]$. Next, recalling that $A=1$, $D=2$ and $r=2$,
\begin{align*}
M_N\leq& 2^p(1-2^{-p})\sum_{n\geq 0}a^{pn}\sum_{\ell\geq 0}2^{-p\ell}  \Big[\mu(G_n^a)\land 
\Big(\frac{2^{d-1}2^{d\ell/2}}{2^{d/2}-1}\sqrt{\frac{\mu(G_n^a)}{N}} \Big)\Big]\leq M_{N,1}+M_{N,2},
\end{align*}
where we separate the cases $n=0$ and $n\geq 1$, i.e.
$$
M_{N,1}=2^p(1-2^{-p})\sum_{\ell\geq 0}2^{-p\ell}  \Big[1\land
\Big(\frac{2^{d-1}2^{d\ell/2}}{2^{d/2}-1}\sqrt{\frac{1}{N}} \Big)\Big]
\leq 2^p(1-2^{-p})\sum_{\ell\geq 0}2^{-p\ell} = 2^p,
$$
and
\begin{align*}
M_{N,2}=& 2^p(1-2^{-p})\sum_{n\geq 1}a^{pn}\sum_{\ell\geq 0}2^{-p\ell}  \Big[\frac{\cM_q^{(\infty)}(\mu)}{a^{q(n-1)}}\land 
\Big(\frac{2^{d-1}2^{d\ell/2}}{2^{d/2}-1}\sqrt{\frac{\cM_q^{(\infty)}(\mu)}{Na^{q(n-1)}}} \Big)\Big]\\
=&2^p(1-2^{-p})a^p\sum_{\ell\geq 0}  2^{-p\ell} \sum_{n\geq 0}a^{pn} \Big[\frac{\cM_q^{(\infty)}(\mu)}{a^{qn}}\land 
\Big(\frac{2^{d-1}2^{d \ell/2}}{2^{d/2}-1}\sqrt{\frac{\cM_q^{(\infty)}(\mu)}{Na^{qn}}}  \Big)\Big]\\
=&2^p(1-2^{-p})a^p\sum_{\ell\geq 0}  2^{-p\ell} \sum_{n\geq 0}a^{(p-q)n} \Big[\cM_q^{(\infty)}(\mu)\land 
\Big(\frac{2^{d-1}2^{d \ell/2}}{2^{d/2}-1}\sqrt{\frac{\cM_q^{(\infty)}(\mu)}{N}}a^{qn/2}  \Big)\Big]\\
=&2^p(1-2^{-p})a^p\cM_q^{(\infty)}(\mu)\sum_{\ell\geq 0}  2^{-p\ell}\Psi_{a,q-p,q/2}
\Big(\frac{2^{d-1}2^{d\ell/2}}{(2^{d/2}-1)\sqrt{N\cM_q^{(\infty)}(\mu)}}\Big).
\end{align*}
By \eqref{tt3} with $r=a$, $\alpha=q-p$ and $\beta=q/2$, recalling that
$\rho_a=a^p[1/(a^{p-q/2}-1)+1/(1-a^{p-q})]$,
\begin{align*}
M_{N,2}\leq & 2^p(1-2^{-p})\cM_q^{(\infty)}(\mu)\rho_a \sum_{\ell\geq 0} 2^{-p\ell} 
\Big(\frac{2^{d-1}2^{d\ell/2}}{(2^{d/2}-1)\sqrt{N\cM_q^{(\infty)}(\mu)}} \Big)^{2(q-p)/q}\\
=& 2^p\rho_a\frac{[\cM_q^{(\infty)}(\mu)]^{p/q}}{N^{(q-p)/q}} 
\Big(\frac{2^{d-1}}{2^{d/2}-1}\Big)^{2(q-p)/q}\frac{1-2^{-p}}{1-2^{d-p-dp/q}},
\end{align*}
observe that $d-p-dp/q<0$ because $q<dp/(d-p)$. All in all, we conclude that
$$
\E[\cT_p^{(\infty)}(\mu_N,\mu)] \leq 2^p\Big(\frac{\e_p}{{\sqrt N}}+1+
\rho_a\frac{[\cM_q^{(\infty)}(\mu)]^{p/q}}{N^{(q-p)/q}}
\Big[\e_p 2^{2p/q-1} +\Big(\frac{2^{d-1}}{2^{d/2}-1}\Big)^{2(q-p)/q}\frac{1-2^{-p}}{1-2^{d-p-dp/q}}\Big] \Big).
$$
For any $\alpha>0$, we may apply this above formula to $\mu^\alpha$, the image measure of $\mu$
by the map $x \mapsto \alpha x$, for which
$\E[\cT_p^{(\infty)}(\mu_N^\alpha,\mu^\alpha)]=\alpha^p\E[\cT_p^{(\infty)}(\mu_N,\mu)]$
and $\cM_q^{(\infty)}(\mu^\alpha)=\alpha^q \cM_q^{(\infty)}(\mu)$. We get
$$
\E[\cT_p^{(\infty)}(\mu_N,\mu)]\! \leq \frac{2^p}{\alpha^p}\Big(\frac{\e_p}{{\sqrt N}}+1+
\rho_a\frac{[\alpha^q \cM_q^{(\infty)}(\mu)]^{p/q}}{N^{(q-p)/q}}
\Big[\e_p 2^{2p/q-1} \!+\Big(\frac{2^{d-1}}{2^{d/2}-1}\Big)^{2(q-p)/q}\hskip-4pt\frac{1-2^{-p}}{1-2^{d-p-dp/q}}\Big] \Big).
$$
Letting $\alpha \to \infty$, we find
$$
\E[\cT_p^{(\infty)}(\mu_N,\mu)]\leq 2^p \frac{[\cM_q^{(\infty)}(\mu)]^{p/q}}{N^{(q-p)/q}}
\rho_a
\Big[\e_p 2^{2p/q-1} +\Big(\frac{2^{d-1}}{2^{d/2}-1}\Big)^{2(q-p)/q}\frac{1-2^{-p}}{1-2^{d-p-dp/q}}\Big] \Big).
$$
Since $\rho_a=a^p/(a^{p-q/2}-1)+a^p/(1-a^{p-q})$ 
and since this result holds for any $a \in (1,\infty)$, the proof is complete.
\end{proof}

\end{document}